\documentclass[10pt,a4paper,reqno,twoside]{amsart}

\newcommand{\usepackageifexists}[1]{%
    \IfFileExists{#1.sty}{\usepackage{#1}}%
       {\GenericInfo{taglia}{Il package #1 non esiste.}}}

\usepackage{verbatim}
\usepackage{amssymb}
\usepackage{upref}
\usepackage{amsmath}
\usepackage{amsthm}
\usepackage{amsfonts}
\usepackage{latexsym}
\usepackage{amssymb}
\usepackage{color}
\usepackage{amsmath}
\usepackage{amsfonts}
\usepackage{amsthm}
\usepackage{amssymb}
\usepackage{amsmath,amsfonts,amscd,bezier}
\usepackage[latin1]{inputenc}

\usepackage{enumerate}

\theoremstyle{plain}
\newtheorem{Thm}{Theorem}
\newtheorem*{Thm*}{Theorem}

\newtheorem{Prop}{Proposition}
\newtheorem{Lem}[Prop]{Lemma}
\newtheorem{Cor}[Prop]{Corollary}
\theoremstyle{definition}

\newtheorem{Hyp}{Hypothesis}

\newtheorem{Not}{Notation}
\theoremstyle{remark}
\newtheorem{Rem}{Remark}

\newtheorem{Example}{Example}

\allowdisplaybreaks

\newcommand{\R}{\mathbb{R}}

\renewcommand{\doteq}{:=}

\def\<#1\>{\left\langle#1\right\rangle }

\begin{document}

\title[The Threshold between Effective and Noneffective Damping for Semilinear Waves]
          {The Threshold between Effective and Noneffective \\
          Damping for Semilinear Wave Equations}
\author[M. D'Abbicco]{Marcello D'Abbicco}

\address{Marcello D'Abbicco, Department of Mathematics, University of Bari, Via E. Orabona 4 - 70125 BARI - ITALY}

\begin{abstract}
In this paper we study the global existence of small data solutions to the Cauchy problem
\[ u_{tt}-\triangle u + \frac\mu{1+t}\,u_t = f(t,u) \,, \qquad u(0,x)=u_0(x)\,, \quad u_t(0,x)=u_1(x)\,,\]
where~$\mu\geq2$. We obtain estimates for the solution and its energy with the same decay rate of the linear problem. We extend our results to a model with polynomial speed of propagation, and to a model with an exponential speed of propagation and a constant damping~$\nu \, u_t$.
\end{abstract}
\maketitle
\baselineskip=15pt


\section{Introduction}\label{sec:Main}

The classical semilinear damped wave equation
\begin{equation}
\label{eq:classic}
\begin{cases}
u_{tt} - \triangle u + u_t= f(u), \qquad t\geq0\,, \ x\in\R^n\,,\\
u(0,x)=u_0(x)\,, \\
u_t(0,x)=u_1(x)\,,
\end{cases}
\end{equation}
has been deeply investigated. In particular, if we assume small, compactly supported data, then by using some linear decay estimates~\cite{Matsu} one can prove that there exists a global solution to~\eqref{eq:classic} if~$p>1+2/n$, and~$p\leq 1+2/(n-2)$ if~$n\geq3$ (see~\cite{TY}). This exponent is \emph{critical}, that is, for suitable nontrivial, arbitrarily small data and~$f(u)=|u|^p$ with~$1<p\leq 1+2/n$, there exists no global solution to~\eqref{eq:classic} (see~\cite{TY, Z}).
\\
If one removes the compactness assumption on the data, still one may obtain global existence for~$p>1+2/n$ if the data are small in the norm of the energy space~$(H^1\times L^2)$ and in the~$L^1$ norm in space dimension~$n=1,2$ (see~\cite{IMN}). In space dimension~$n\geq3$ the compactness assumption on the data may be replaced by assuming that the data are small in the energy space with a suitable weight~\cite{IT}.
\\
On the other hand, weakening the assumption of smallness replacing the $L^1$ norm of the data with the $L^m$ norm for some~$m\in(1,2)$, the \emph{critical} exponent becomes~$1+2m/n$ (see~\cite{IO}). In particular, one obtains~$1+4/n$ if the smallness is only taken in the energy space, without additional~$L^m$ regularity or compact support assumption. The same exponent was first obtained in~\cite{NO} by using a \emph{modified potential well} technique.

It has been recently proved~\cite{DALR} that the exponent~$1+2/n$ remains \emph{critical} if we consider the wave equation with a time-dependent \emph{effective} damping $b(t)u_t$ satisfying suitable assumptions. We say that the damping term is \emph{effective} for the wave equation if the linear estimates have the same decay rate of the corresponding heat equation $b(t)u_t-\triangle u=0$ (see~\cite{W05, W07, Ya, Ya1}). In fact, the exponent~$1+2/n$ was first proved to be critical by Fujita for the semilinear heat equation~\cite{Fuj}.
\\
In the special case~$b(t)=\mu\,(1+t)^{-k}$, the dissipation is \emph{effective} for any~$\mu>0$, if~$|\kappa|<1$. In this special case, a global existence result has been obtained in~\cite{LNZ, N}. On the other hand, if~$b(t)$ is a sufficiently smooth function satisfying~$\limsup_{t\to\infty}tb(t)<1$ then the dissipation is \emph{non effective}~\cite{W06}. The case~$b(t)=\mu(1+t)^{-1}$ with~$\mu\geq1$ is more difficult to manage, since the dissipation is \emph{effective} for large~$\mu$ and \emph{noneffective} for small~$\mu$. The precise threshold depends on which type of estimate one is studying.
\\
Completely different effects appear if one consider a space-dependent damping term~\cite{ITY, ITY11, Nx} or a time-space dependent damping term~\cite{LNZ11, Wakatx}; in this case the exponent for the global existence changes accordingly to the decay in the space variable.

\bigskip

In this paper, we consider the Cauchy problem
\begin{equation}
\label{eq:diss1}
\begin{cases}
u_{tt}-\Delta u+ \frac\mu{1+t} \, u_t= f(t,u), & t\geq0, \ x\in\R^n\,,\\
u(0,x)=u_0(x)\,, \\
u_t(0,x)=u_1(x)\,.
\end{cases}
\end{equation}
%
%
\begin{Hyp}\label{Hyp:1fu}
We assume that
\begin{equation}\label{eq:1disscontr}
f(t,0)=0\,, \qquad \text{and} \quad |f(t,u)-f(t,v)|\lesssim (1+t)^\gamma\,|u-v|(|u|+|v|)^{p-1} \,,
\end{equation}
for some~$\gamma\geq-2$ and~$p>1$, satisfying~$p\leq 1+2/(n-2)$ if~$n\geq3$.
\end{Hyp}
\begin{Not}\label{Not:1}
We will use the following notation.
\begin{itemize}
\item We say that there exists \emph{a solution to}~\eqref{eq:diss1}, if there exists a unique
\[ u\in\mathcal{C}([0,\infty),H^1)\cap\mathcal{C}^1([0,\infty), L^2)\,,\]
global solution to~\eqref{eq:diss}, in a weak sense.
\item We refer to
\[ \|(\nabla u,u_t)(t,\cdot)\|_{L^2}^2 \doteq \|\nabla u(t,\cdot)\|_{L^2}^2 + \|u_t(t,\cdot)\|_{L^2}^2\,, \]
as the energy of the solution to~\eqref{eq:diss1}.
\item For any~$m\in[1,2)$ we define
\[ \mathcal{D}_m\doteq (L^m\cap H^1) \times (L^m\cap L^2)\,, \qquad \|(u,v)\|_{\mathcal{D}_m}^2 \doteq \|u\|_{L^m}^2+\|u\|_{H^1}^2+\|v\|_{L^m}^2+\|v\|_{H^1}^2 \,.\]
\end{itemize}
\end{Not}
%
%
For the ease of reading, we collect our main results them in three separate theorems.
\begin{Thm}\label{Thm:1L2}
Let~$n\geq1$, $\mu\geq2$ and~$p>1+2(2+\gamma)/n$. Then there exists~$\epsilon>0$ such that for any initial data
\begin{equation}\label{eq:dataL2}
(u_0,u_1)\in H^1\times L^2\,, \qquad \text{satisfying}\quad \|(u_0,u_1)\|_{H^1\times L^2}\leq\epsilon\,,
\end{equation}
there exists a solution to~\eqref{eq:diss1}. Moreover, the solution and its energy satisfy the estimates
\begin{align}
\label{eq:1uL2}
\|u(t,\cdot)\|_{L^2} & \lesssim \, \|(u_0,u_1)\|_{H^1\times L^2} \,,\\
\label{eq:1enL2}
\|(\nabla u,u_t)(t,\cdot)\|_{L^2} & \lesssim (1+t)^{-1} \, \|(u_0,u_1)\|_{H^1\times L^2} \,.
\end{align}
\end{Thm}
%
%
\begin{Thm}\label{Thm:1L1}
Let~$n\leq 4$, $\mu \geq n+2$ and
\[ p>1+(2+\gamma)/n\,,\]
if~$\gamma\geq n-2$, or~$p\geq2$ otherwise. Then there exists~$\epsilon>0$ such that for any initial data
\begin{equation}\label{eq:dataL1}
(u_0,u_1)\in \mathcal{D}_1\,, \qquad \text{satisfying}\quad \|(u_0,u_1)\|_{\mathcal{D}_1}\leq\epsilon\,,
\end{equation}
there exists a solution to~\eqref{eq:diss1}. Moreover, the solution and its energy satisfy the decay estimates
\begin{align}
\label{eq:1uL1}
\|u(t,\cdot)\|_{L^2}
    & \lesssim (1+t)^{-\frac{n}2} \, \|(u_0,u_1)\|_{\mathcal{D}_1} \,,\\
\label{eq:1enL1}
\|(\nabla u,u_t)(t,\cdot)\|_{L^2}
    & \lesssim \begin{cases}
    (1+t)^{-\frac{n}2-1} \, \|(u_0,u_1)\|_{\mathcal{D}_1} & \text{if~$\mu>n+2$,}\\
    (1+t)^{-\frac\mu2}\log(e+t) \, \|(u_0,u_1)\|_{\mathcal{D}_1} & \text{if~$\mu=n+2$.}
    \end{cases}
\end{align}
\end{Thm}
The exponent~$1+(2+\gamma)/n$ in Theorem~\ref{Thm:1L1} can be proved to be \emph{critical} by using a \emph{modified} test function method, that is, there exists no global solution to~\eqref{eq:diss1} if~$p\leq 1+(2+\gamma)/n$, for suitable data, arbitrarily small in~$\mathcal{D}_1$ (see Example~2 in~\cite{DAL}). 

Theorem~\ref{Thm:1L1} is a special case of the following.
\begin{Thm}\label{Thm:1Lm}
Let~$m\in[1,2)$, $n\leq 4/(2-m)$,
\begin{align}
\label{eq:mu}
\mu
	& \geq 2+ n\left(\frac2m-1\right)\,, \quad \text{and} \\
\label{eq:p}
p
    & > 1+\frac{m(2+\gamma)}n \,,
\end{align}
if~$\gamma+2\geq n(2-m)/m^2$, or $p\geq 2/m$ otherwise. Then there exists~$\epsilon>0$ such that for any initial data
\begin{equation}\label{eq:dataLm}
(u_0,u_1)\in \mathcal{D}_m\,, \qquad \text{satisfying}\quad \|(u_0,u_1)\|_{\mathcal{D}_m}\leq\epsilon\,,
\end{equation}
there exists a solution to~\eqref{eq:diss1}. Moreover, the solution and its energy satisfy the decay estimates
\begin{align}
\label{eq:1uLm}
\|u(t,\cdot)\|_{L^2}
    & \lesssim \,(1+t)^{-n\left(\frac1m-\frac12\right)} \, \|(u_0,u_1)\|_{\mathcal{D}_m} \,,\\
\label{eq:1enLm}
\|(\nabla u,u_t)(t,\cdot)\|_{L^2}
    & \lesssim \begin{cases}
    (1+t)^{-n\left(\frac1m-\frac12\right)-1} \, \|(u_0,u_1)\|_{\mathcal{D}_m} & \text{if~$\mu>2+n(2/m-1)$,} \\
    (1+t)^{-\frac\mu2}\log(e+t) \, \|(u_0,u_1)\|_{\mathcal{D}_m} & \text{if~$\mu=2+n(2/m-1)$.}
    \end{cases}
\end{align}
\end{Thm}
\begin{Rem}\label{Rem:range}
We recall that in space dimension~$n\geq3$ we assumed~$p\leq 1+2/(n-2)$ in Hypothesis~\ref{Hyp:1fu}.

For~$n\geq3$, the set~$(1+2(2+\gamma)/n, 1+2/(n-2)]$ of the global existence in Theorem~\ref{Thm:1L2} is nonempty if, and only if, either $\gamma\in[-2,-1]$, or~$\gamma\in(-1,1)$ and $n<2(2+\gamma)/(1+\gamma)$. 

For~$n=3$, the range of admissible exponents~$p$ for the global existence in Theorem~\ref{Thm:1L1} is nonempty if, and only if, $\gamma<4$. We have the range~$(1+(2+\gamma)/3,3]$ if~$\gamma\in[1,4)$, and the range~$[2,3]$ if~$\gamma\in[-2,1)$. For~$n=4$ we only have the admissible exponent~$p=2$, provided that~$\gamma<2$.

More in general, for any~$m\in[1,2)$ there exists~$\overline{n}=\overline{n}(m,\gamma)\geq3$ such that the range of admissible exponents is empty for~$n\geq \overline{n}$. If~$\gamma\in[-2,-1]$ then $\overline{n}(m,\gamma)\to\infty$ as~$m\to2$.
\end{Rem}
\begin{Rem}\label{Rem:comparison}
Let us assume~$\mu\geq n+2$ and let the data verify condition~\eqref{eq:dataL1}. We may compare Theorems~\ref{Thm:1L2}, \ref{Thm:1L1} and~\ref{Thm:1Lm}, looking for the largest range of admissible exponents~$p$. Indeed, due to the bound~$p\geq2$ in Theorem~\ref{Thm:1L1}, we may get benefit by applying Theorem~\ref{Thm:1Lm} for some~$m\in(1,2)$, or even Theorem~\ref{Thm:1L2}.
\\
Let us fix~$n\geq1$. If~$\gamma\geq n-2$, then the range in Theorem~\ref{Thm:1L1} cannot be further improved, i.e we get
\[ p \in \begin{cases}
(1+(2+\gamma)/n,\infty) & \text{if~$n=1,2$ and~$\gamma\geq n-2$,}\\
(1+(2+\gamma)/3,3] & \text{if~$n=3$ and~$\gamma\in[3,4)$.}
\end{cases}\]
If~$\gamma\in(-2,n-2)$, let~$m\in(1,2)$ be the largest solution to
\[ \left(\frac{2+\gamma}n\right) m^2+m-2 = 0 \,. \]
In correspondence of this~$m=m(n,\gamma)$, we obtain the range in Theorem~\ref{Thm:1Lm}, i.e. either $p>(1+(2+\gamma)m/n$ if~$n=1,2$ or $p\in(1+(2+\gamma)m/n,1+2/(n-2)]$, for any~$n\geq3$ which makes the interval nonempty.
\\
Finally, if~$\gamma=-2$ we obtain either the range~$p>1$ if~$n=1,2$, or the range $p\in(1,1+2/(n-2))$ if~$n\geq3$, by applying Theorem~\ref{Thm:1L2}.
\end{Rem}
If~$\mu\in(2,n+2)$, we may apply Theorem~\ref{Thm:1Lm} only for~$m\in [\ell,2)$, where
\begin{equation}\label{eq:ell}
\ell=\ell(n,\mu) \doteq \frac{2n}{n+\mu-2} \,.
\end{equation}
In particular, setting~$m=\ell$ we immediately have the following.
\begin{Cor}\label{Cor:1ell}
Let~$n\geq1$ and $\mu\in(2,2+n)$, and let us assume
\begin{align}
\label{eq:pell}
p
    & > 1+\frac{2(2+\gamma)}{n+\mu-2} \,,\\
\nonumber
\text{if} \quad \gamma
    & \geq \frac{(\mu-2)(n+\mu-2)}{2n} \, -2\,,
\end{align}
or $p\geq 1 + (\mu-2)/n$ otherwise. Let~$\ell=\ell(n,\mu)$ be defined as in~\eqref{eq:ell}. Then there exists~$\epsilon>0$ such that for any initial data
\begin{equation}\label{eq:dataLell}
(u_0,u_1)\in \mathcal{D}_\ell\,, \qquad \text{satisfying} \quad \|(u_0,u_1)\|_{\mathcal{D}_\ell}\leq\epsilon\,,
\end{equation}
there exists a solution to~\eqref{eq:diss1}. Moreover, the solution and its energy satisfy the decay estimates 
\begin{align}
\label{eq:1uell}
\|u(t,\cdot)\|_{L^2}
    & \lesssim (1+t)^{-(\frac\mu2-1)} \, \|(u_0,u_1)\|_{\mathcal{D}_\ell} \,,\\
\label{eq:1enell}
\|(\nabla u,u_t)(t,\cdot)\|_{L^2}
    & \lesssim (1+t)^{-\frac\mu2}\log(e+t) \, \|(u_0,u_1)\|_{\mathcal{D}_\ell} \,.
\end{align}
\end{Cor}

\section{Models with time-dependent speed}\label{sec:extension}

More in general, one may investigate on the global existence for a wave equation with time-dependent propagation speed
\begin{equation}
\label{eq:diss}
\begin{cases}
u_{tt}-\lambda(t)^2\triangle u+ b(t) u_t= f(t,u), & t\geq0, \ x\in\R^n\,,\\
u(0,x)=u_0(x)\,, \\
u_t(0,x)=u_1(x)\,,
\end{cases}
\end{equation}
expecting interactions between the speed~$\lambda(t)$ and the damping coefficient~$b(t)$. In this setting, one may still classify the dissipation produced by the damping term in \emph{effective} and \emph{non effective}, with respect to the speed and to the considered estimate (see~\cite{DAEm,DAE}). 
In particular, we are interested in the following two models.
\begin{Example}[Polynomial speed]\label{Ex:pol}
Let~$\lambda(t)=(1+t)^{q-1}$ for some~$q>0$, and~$b(t)=\nu(1+t)^{-1}$ for some~$\nu\in\R$, that is,
\begin{equation}\label{eq:disspol}
\begin{cases}
u_{tt}-(1+t)^{2(q-1)}\triangle u+ \frac\nu{1+t}\, u_t= f(t,u), & t\geq0, \ x\in\R^n\,,\\
u(0,x)=u_0(x)\,, \\
u_t(0,x)=u_1(x)\,.
\end{cases}
\end{equation}
With respect to this model, we will denote~$\Lambda(t)=(1+t)^q/q$, and
\[ \mu=\mu(\nu,q) \doteq \frac{\nu-1}q +1\,. \]
We remark that for~$q=1$ we find again~\eqref{eq:diss1} and~$\nu=\mu$.
\end{Example}
\begin{Example}[Exponential speed]\label{Ex:exp}
Let~$\lambda(t)=e^{rt}$ for some~$r>0$ and~$b=\nu$ for some~$\nu\in\R$, that is,
\begin{equation}\label{eq:dissexp}
\begin{cases}
u_{tt}-e^{2rt}\triangle u+ \nu\, u_t= f(t,u), & t\geq0, \ x\in\R^n\,,\\
u(0,x)=u_0(x)\,, \\
u_t(0,x)=u_1(x)\,.
\end{cases}
\end{equation}
With respect to this model, we will denote~$\Lambda(t)=e^{rt}/r$, and
\[ \mu=\mu(\nu) \doteq \nu +1\,. \]
\end{Example}
To deal with both models in Examples~\ref{Ex:pol} and~\ref{Ex:exp}, we modify the assumption on~$f(t,u)$.
\begin{Hyp}\label{Hyp:fu}
We assume that the nonlinear term in~\eqref{eq:diss} satisfies
\begin{equation}\label{eq:disscontr}
f(t,0)=0\,, \qquad |f(t,u)-f(t,v)|\lesssim \lambda(t)^2\Lambda(t)^\gamma\,|u-v|(|u|+|v|)^{p-1} \,,
\end{equation}
for some~$\gamma\geq-2$ and for a given $p>1$, satisfying~$p\leq 1+2/(n-2)$ if~$n\geq3$.
\end{Hyp}
With the notation in Examples~\ref{Ex:pol} and~\ref{Ex:exp}, the inequality in condition~\eqref{eq:disscontr} may be explicitated by means of the time-dependent speed~$\lambda(t)$ and its anti-derivative~$\Lambda(t)$, giving
\begin{align}
\label{eq:contrpol}
|f(t,u)-f(t,v)|
	& \lesssim (1+t)^{(\gamma+2) q-2}\,|u-v|(|u|+|v|)^{p-1} \,,\\
\label{eq:contrexp}
|f(t,u)-f(t,v)|
	& \lesssim e^{(\gamma+2)rt}\,|u-v|(|u|+|v|)^{p-1} \,.
\end{align}
To state our results, we still use Notation~\ref{Not:1} but now we refer to
\[ \|(\lambda\nabla u,u_t)(t,\cdot)\|_{L^2}^2 \doteq \lambda(t)^2\,\|\nabla u(t,\cdot)\|_{L^2}^2 + \|u_t(t,\cdot)\|_{L^2}^2\,, \]
as the energy of the solution to~\eqref{eq:diss}.
\begin{Thm}\label{Thm:L2}
Let~$n\geq1$, $\mu\geq2$ and~$p>1+2(2+\gamma)/n$. Then there exists~$\epsilon>0$ such that, for any initial data as in~\eqref{eq:dataL2} there exists a solution to~\eqref{eq:diss}. Moreover, the solution and its energy satisfy the estimates
\begin{align}
\label{eq:uL2}
\|u(t,\cdot)\|_{L^2} & \lesssim \|(u_0,u_1)\|_{H^1\times L^2} \,,\\
\label{eq:enL2}
\|(\lambda\nabla u,u_t)(t,\cdot)\|_{L^2} & \lesssim \lambda(t)\,\Lambda(t)^{-1} \, \|(u_0,u_1)\|_{H^1\times L^2} \,.
\end{align}
\end{Thm}
\begin{Thm}\label{Thm:L1m}
Let~$m\in[1,2)$ and~$n\leq 4/(2-m)$. Let us assume~\eqref{eq:mu}, and~\eqref{eq:p} if~$\gamma+2\geq n(2-m)/m^2$, or $p\geq 2/m$ otherwise. Then there exists~$\epsilon>0$ such that, for any initial data as in~\eqref{eq:dataLm} there exists a solution to~\eqref{eq:diss}. Moreover, the solution and its energy satisfy the estimates
\begin{align}
\label{eq:uLm}
\|u(t,\cdot)\|_{L^2}
    & \lesssim \Lambda(t)^{-n\left(\frac1m-\frac12\right)} \, \|(u_0,u_1)\|_{\mathcal{D}_m} \,,\\
\label{eq:enLm}
\|(\lambda\nabla u,u_t)(t,\cdot)\|_{L^2}
    & \lesssim \begin{cases}
    \lambda(t)\,\Lambda(t)^{-n\left(\frac1m-\frac12\right)-1} \, \|(u_0,u_1)\|_{\mathcal{D}_m} & \text{~$\mu>2+n(2/m-1)$,} \\
    \lambda(t)\,\Lambda(t)^{-\frac\mu2}\log(e+\Lambda(t)) \, \|(u_0,u_1)\|_{\mathcal{D}_m} & \text{if~$\mu=2+n(2/m-1)$.}
    \end{cases}
\end{align}
\end{Thm}
In the polynomial case the exponent~$1+(2+\gamma)/n$ obtained in Theorem~\ref{Thm:L1m} for~$m=1$ can be proved to be \emph{critical} by using a \emph{modified} test function method. Indeed, thanks to Theorem~1 in~\cite{DAL}, there exists no global solution to~\eqref{eq:diss1} if~$p\leq 1+(2+\gamma)/n$, for suitable, arbitrarily small data in~$L^1$.
\begin{Rem}
Taking~$\lambda(t)=(1+t)^{q-1}$ as in Example~\ref{Ex:pol} or, respectively, $\lambda(t)=e^{rt}$ as in Example~\ref{Ex:exp}, estimates \eqref{eq:uL2}-\eqref{eq:enL2} may be written in the form
\begin{align*}
\|u(t,\cdot)\|_{L^2} & \lesssim  \|(u_0,u_1)\|_{H^1\times L^2} \,,\\
\|\nabla u(t,\cdot)\|_{L^2} & \lesssim (1+t)^{-q} \, \|(u_0,u_1)\|_{H^1\times L^2} \,,\\
\|u_t(t,\cdot)\|_{L^2} & \lesssim (1+t)^{-1} \, \|(u_0,u_1)\|_{H^1\times L^2} \,,
\intertext{or, respectively,}
%
\|u(t,\cdot)\|_{L^2} & \lesssim \|(u_0,u_1)\|_{H^1\times L^2} \,,\\
\|\nabla u(t,\cdot)\|_{L^2} & \lesssim e^{-rt} \, \|(u_0,u_1)\|_{H^1\times L^2} \,,\\
\|u_t(t,\cdot)\|_{L^2} & \lesssim \|(u_0,u_1)\|_{H^1\times L^2} \,.
\end{align*}
Estimates \eqref{eq:uLm}-\eqref{eq:enLm} may be similarly written, including the additional decay rate $(1+t)^{-n\left(\frac1m-\frac12\right)\,q}$ or, respectively, $e^{-n\left(\frac1m-\frac12\right)\,rt}$.
\end{Rem}
\begin{Cor}\label{Cor:ell}
Let~$n\geq1$ and $\mu, p$ be as in Corollary~\ref{Cor:1ell}. Then there exists~$\epsilon>0$ such that for any initial data as in~\eqref{eq:dataLell} there exists a solution to~\eqref{eq:diss}. Moreover, the solution and its energy satisfy the decay estimates \eqref{eq:uLm}-\eqref{eq:enLm} with~$m=\ell$, that is,
\begin{align}
\label{eq:uell}
\|u(t,\cdot)\|_{L^2}
    & \lesssim \Lambda(t)^{-(\frac\mu2-1)} \, \|(u_0,u_1)\|_{\mathcal{D}_\ell} \,,\\
\label{eq:enell}
\|(\lambda\nabla u,u_t)(t,\cdot)\|_{L^2}
    & \lesssim \lambda(t)\,\Lambda(t)^{-\frac\mu2}\log(e+\Lambda(t)) \, \|(u_0,u_1)\|_{\mathcal{D}_\ell} \,.
\end{align}
\end{Cor}
Theorems~\ref{Thm:L2} and~\ref{Thm:L1m} still hold if we consider a more general propagation speed, provided that we take a damping term in a suitable form.
\begin{Hyp}\label{Hyp:lambda}
We assume that~$\lambda\in\mathcal{C}^1$, with~$\lambda(t)>0$ for any~$t\geq0$ and~$\lambda\not\in L^1$. Let
\[ \Lambda(t)\doteq \lambda_0 + \int_0^t \lambda(\tau)\,d\tau\,, \]
for some~$\lambda_0>0$, be an anti-derivative of~$\lambda(t)$. We assume that
\begin{equation}\label{eq:bmuprime}
b(t) \doteq \mu \,\frac{\lambda(t)}{\Lambda(t)} - \frac{\lambda'(t)}{\lambda(t)} \,,
\end{equation}
for some~$\mu>0$, for any~$t\geq0$.
\end{Hyp}
We remark that~$\Lambda(t)$ is a strictly positive, strictly increasing function such that~$\Lambda(t)\to\infty$ as~$t\to\infty$. The assumption~$\lambda\not\in L^1$ which guarantees this latter property was first used in~\cite{DAR, DAR2} to derive energy estimates in the setting of linear systems, eventually with the presence of a dissipative lower order term. On the other hand, if we consider the equation
\[ u_{tt}-\lambda(t)^2\triangle u +b(t)u_t = 0\,, \]
then still a dissipative effect on the energy~$\|(\lambda \nabla u, u_t)\|_{L^2}$ appears (see~\cite{DAE}), provided that
\begin{equation}\label{eq:dissDAE}
\frac{\lambda'(t)}{\lambda(t)} + b(t) \geq 0\,.
\end{equation}
We notice that~\eqref{eq:dissDAE} reduces to~$\lambda'(t)\geq0$ if~$b\equiv0$ (see~\cite{H-W}). Dealing with~\eqref{eq:diss}, thanks to the special structure of~$b(t)$ given by~\eqref{eq:bmuprime} we see that~\eqref{eq:dissDAE} is satisfied for any~$\mu\geq0$.
\begin{Rem}
It is clear that Hypothesis~\ref{Hyp:lambda} is consistent with the notation used in Examples~\ref{Ex:pol} and~\ref{Ex:exp}. On the other hand, polynomial and exponential speeds in Examples~\ref{Ex:pol} and~\ref{Ex:exp} have the following property: there exists an anti-derivative~$\Lambda(t)$ of~$\lambda(t)$ and a constant~$\alpha\in\R$ such that
\begin{equation}\label{eq:lambdacomp}
\frac{\lambda'(t)}{\lambda(t)} = \alpha \,\frac{\lambda(t)}{\Lambda(t)} \,.
\end{equation}
Property~\eqref{eq:lambdacomp} means that if~$b(t)=\nu\lambda(t)/\Lambda(t)$ for some~$\nu\in\R$, then~\eqref{eq:bmuprime} holds with~$\mu = \nu+\alpha$. This constant is~$\alpha=(q-1)/q$ in Example~\ref{Ex:pol} and~$\alpha=1$ in Example~\ref{Ex:exp}. We notice that~\eqref{eq:lambdacomp} is equivalent to say~$\lambda(t)=C\,\Lambda(t)^\alpha$, for some~$C>0$.
\end{Rem}
Theorems \ref{Thm:1L2}-\ref{Thm:1L1}-\ref{Thm:1Lm} immediately follow as a consequence of Theorems \ref{Thm:L2}-\ref{Thm:L1m}, which we will prove in Section~\ref{sec:well} for a general propagation speed and for the related dissipation, satisfying Hypothesis~\ref{Hyp:lambda}.

\section{Linear Estimates}\label{sec:linear}

In order to prove our results we will apply Duhamel's principle. Therefore, we derive estimates for the family of parameter-dependent linear Cauchy problems:
\begin{equation}\label{eq:CPlin}
\begin{cases}
v_{tt}-\lambda(t)^2\triangle v+ b(t)\, v_t= 0, & t\geq s, \ x\in\R^n\,,\\
v(s,x)=v_0(x)\,, \\
v_t(s,x)=v_1(x)\,.
\end{cases}
\end{equation}
%
%
\begin{Lem}\label{Lem:CPlin2}
Let~$(v_0,v_1)\in L^2\times L^2$. If $\mu\geq1$ then the solution to~\eqref{eq:CPlin} satisfies the estimate
\begin{align}
\label{eq:decayv2}
\|v(t,\cdot)\|_{L^2}
    & \lesssim \|v_0\|_{L^2} + \frac{\Lambda(s)}{\lambda(s)} \, \|v_1\|_{L^2} \,.
\intertext{Moreover, if~$(v_0,v_1)\in H^1 \times L^2$ and $\mu\geq2$, then the energy of the solution to~\eqref{eq:CPlin} satisfies the estimate}
\label{eq:decayenv2}
\|(\lambda\nabla v,v_t)(t,\cdot)\|_{L^2}
    & \lesssim \frac{\lambda(t)}{\Lambda(t)}\,\Lambda(s) \,\left( \|v_0\|_{H^1} + \frac1{\lambda(s)} \, \|v_1\|_{L^2}\right) \,.
\end{align}
\end{Lem}
\begin{Lem}\label{Lem:CPlinm}
Let~$(v_0,v_1)\in L^m\cap L^2$ for some~$m\in [1,2)$. If $\mu\geq1$ and $\mu > n(2/m-1)$ then the solution to~\eqref{eq:CPlin} satisfies the estimate
\begin{align}
\label{eq:decayvm}
\|v(t,\cdot)\|_{L^2}
    & \lesssim \Lambda(t)^{-n\left(\frac1m-\frac12\right)} \, \left\{ \|v_0\|_{L^m} + \frac{\Lambda(s)}{\lambda(s)} \, \|v_1\|_{L^m} + \Lambda(s)^{n\left(\frac1m-\frac12\right)} \Bigl(\|v_0\|_{L^2} + \frac{\Lambda(s)}{\lambda(s)} \, \|v_1\|_{L^2} \Bigr)\right\} \,,
%
\intertext{whereas if~$\mu=n(2/m-1)\geq1$ it satisfies the estimate}
\label{eq:decayvcrit}
\|v(t,\cdot)\|_{L^2}
    & \lesssim \, \Lambda(t)^{-\frac\mu2} \, \log \left(1+\frac{\Lambda(t)}{\Lambda(s)}\right) \, \left\{ \|v_0\|_{L^m} + \frac{\Lambda(s)}{\lambda(s)} \, \|v_1\|_{L^m} + \Lambda(s)^{\frac\mu2}\Bigl( \|v_0\|_{L^2} + \frac{\Lambda(s)}{\lambda(s)} \, \|v_1\|_{L^2} \Bigl) \right\} \,.
%
\end{align}
Moreover, if~$(v_0,v_1)\in \mathcal{D}_m$ and $\mu > 2 + n(2/m-1)$ then the energy of the solution to~\eqref{eq:CPlin} satisfies the estimate
\begin{multline}\label{eq:decayenvm}
\|(\lambda\nabla v,v_t)(t,\cdot)\|_{L^2} \lesssim \lambda(t)\,\Lambda(t)^{-n\left(\frac1m-\frac12\right)-1} \, \left\{ \|v_0\|_{L^m} + \frac{\Lambda(s)}{\lambda(s)} \, \|v_1\|_{L^m} \right. \\
+ \left. \Lambda(s)^{n\left(\frac1m-\frac12\right)+1}\Bigl( \|v_0\|_{H^1} + \frac1{\lambda(s)} \, \|v_1\|_{L^2} \Bigl) \right\} \,,
\end{multline}
%
whereas if~$\mu=2+n(2/m-1)$ it satisfies the estimate
\begin{multline}\label{eq:decayenvcrit}
\|(\lambda\nabla v,v_t)(t,\cdot)\|_{L^2}\lesssim \,\lambda(t)\,\Lambda(t)^{-\frac\mu2} \, \log \left(1+\frac{\Lambda(t)}{\Lambda(s)}\right) \, \left\{ \|v_0\|_{L^m} + \frac{\Lambda(s)}{\lambda(s)} \, \|v_1\|_{L^m} \right. \\
\left. + \Lambda(s)^{\frac\mu2} \Bigr( \|v_0\|_{H^1} + \frac1{\lambda(s)} \, \|v_1\|_{L^2} \Bigl) \right\} \,.
\end{multline}
\end{Lem}
We recall that taking~$\lambda(t)=1$, $\Lambda(t)=1+t$ and~$b(t)=\mu(1+t)^{-1}$ we obtain the linear estimates corresponding to~\eqref{eq:diss1}.
\begin{Rem}
Since~\eqref{eq:CPlin} is linear, we may write the solution to~\eqref{eq:CPlin} into the form
\begin{equation}\label{eq:vMult}
v (t,x) = E_0(t,s,x) \ast_{(x)} v_0(x) + E_1(t,s,x) \ast_{(x)} v_1(x) \,.
\end{equation}
\end{Rem}


The estimates in Lemmas~\ref{Lem:CPlin2} and~\ref{Lem:CPlinm} are deeply related to the special structure of the equation in~\eqref{eq:CPlin}. To prove them we follow the approach used in~\cite{W04} to derive $L^2-L^2$ estimates for the linear damped wave equation 
\[ u_{tt}-\triangle u + \frac\mu{1+t}\,u_t=0\,, \qquad t\geq0\,, \]
modifying it to derive $(L^m\cap L^2)-L^2$ estimates, and taking into account the presence of the parameter~$s$ and of the speed~$\lambda(t)$.
\\
%
Let us put~$w(\Lambda(t)\,|\xi|)=\widehat{v}(t,\xi)$, and let us denote~$\tau=\Lambda(t)\,|\xi|$ and $\sigma=\Lambda(s)\,|\xi|$. Then~$\sigma>0$ for any~$\xi\neq0$, and from the equation in~\eqref{eq:CPlin} we obtain the ordinary differential equation
\begin{equation}
\label{eq:scale}
w'' + w + \frac\mu\tau \, w' =0, \qquad \tau\geq \sigma\,.
\end{equation}
If we put~$ \rho \doteq (1-\mu)/2$ and~$w(\tau)=\tau^\rho\,y(\tau)$ then from~\eqref{eq:scale} we obtain the Bessel's differential equation of order~$\pm\rho$:
\begin{equation}\label{eq:yBess}
\tau^2 y'' + \tau y' + (\tau^2-\rho^2) y =0 \,, \qquad \tau\geq\sigma\,.
\end{equation}
%
%
A system of linearly independent solution to~\eqref{eq:yBess} is given by the pair of Hankel functions~$\mathcal{H}_\rho^\pm(\tau)$, hence we put
\[ w^\pm(\tau)\doteq \tau^\rho \mathcal{H}_\rho^\pm(\tau) \,. \]
%
If we define
\begin{align}
\label{eq:PsiH}
\Psi_{k,r,\delta} (t,s,|\xi|)
	& \doteq \frac{i\pi}4\, |\xi|^k\, \det \begin{pmatrix}
\mathcal{H}_r^-\bigl(\Lambda(s)\,|\xi|\bigr) & \mathcal{H}_{r+\delta}^-\bigl(\Lambda(t)\,|\xi|\bigr) \\
\mathcal{H}_r^+\bigl(\Lambda(s)\,|\xi|\bigr) & \mathcal{H}_{r+\delta}^+\bigl(\Lambda(t)\,|\xi|\bigr)
\end{pmatrix} \\
\label{eq:PsiI}
	& \equiv -\frac\pi2 \, \csc(\rho\pi) \, |\xi|^k\, \det \begin{pmatrix}
\mathcal{I}_{-r}^-\bigl(\Lambda(s)\,|\xi|\bigr) & \mathcal{I}_{-(r+\delta)}^-\bigl(\Lambda(t)\,|\xi|\bigr) \\
(-1)^{|\delta|}\mathcal{I}_r^+\bigl(\Lambda(s)\,|\xi|\bigr) & \mathcal{I}_{r+\delta}^+\bigl(\Lambda(t)\,|\xi|\bigr)
\end{pmatrix}\,,
\end{align}
then the solution to~\eqref{eq:CPlin} is given by
\[ \widehat{v} (t,\xi) = \Phi_0(t,s,\xi) \widehat{v_0}(\xi) + \Phi_1(t,s,\xi) \widehat{v_1}(\xi) \,, \]
that is, $\Phi_j(t,s,\xi)$ is the Fourier transform of~$E_j(t,s,x)$ introduced in~\eqref{eq:vMult}.
We may now write the multipliers and their time-derivatives in the form
\begin{align}
\label{eq:Phi0}
\Phi_0(t,s,\xi)
	& = \frac{\Lambda(t)^\rho}{\Lambda(s)^{\rho-1}} \, \Psi_{1,\rho-1,1}\,,\\
\label{eq:Phi1}
\Phi_1(t,s,\xi)
	& = -\frac1{\lambda(s)}\,\frac{\Lambda(t)^\rho}{\Lambda(s)^{\rho-1}} \, \Psi_{0,\rho,0}\,,\\
\label{eq:Phi0t}
\partial_t\Phi_0(t,s,\xi)
	& = \lambda(t)\,\frac{\Lambda(t)^\rho}{\Lambda(s)^{\rho-1}} \, \Psi_{2,\rho-1,0}\,,\\
\label{eq:Phi1t}
\partial_t\Phi_1(t,s,\xi)
	& = -\frac{\lambda(t)}{\lambda(s)}\,\frac{\Lambda(t)^\rho}{\Lambda(s)^{\rho-1}} \, \Psi_{1,\rho,-1}\,,
\end{align}
Let us fix~$K\in(0,1)$, independent on $s$ and $t$. The following three properties hold:
\begin{align}
\label{eq:Hlarge}
|\mathcal{H}_\nu^\pm(\tau)|
	& \lesssim \tau^{-1/2} \,, \quad \text{for~$\tau\in[K,\infty)$,}\\
\label{eq:Hsmall}
|\mathcal{H}_\nu^\pm(\tau)|
	& \lesssim \begin{cases}
        \tau^{-|\nu|} \,, & \text{for~$\tau\in(0,K]$ if~$\nu\neq0$,}\\
        -\log \tau \,, & \text{for~$\tau\in(0,K]$ if~$\nu=0$,}
        \end{cases}\\
\label{eq:I}
|\mathcal{I}_\nu^\pm(\tau)|
	& \lesssim \tau^{\nu} \,, \quad \text{for~$\tau\in(0,\infty)$.}
\end{align}
According to the parameter~$s\geq0$ and to the variable~$t\geq s$, we divide the frequencies in three intervals:
\[
I_1 \doteq \left\{ |\xi|\geq \frac{K}{\Lambda(s)} \right\} \,,\qquad I_2 \doteq \left\{ \frac{K}{\Lambda(s)} \geq |\xi| \geq \frac{K}{\Lambda(t)} \right\} \,,\qquad I_3 \doteq \left\{  \frac{K}{\Lambda(t)} \geq |\xi|\right\} \,.
\]
%
%
We are now ready to prove our linear estimates.
\begin{proof}[Proof of Lemma~\ref{Lem:CPlin2}]
By virtue of Parseval's identity, to derive $L^2-L^2$ estimates for the solution to~\eqref{eq:CPlin} and its energy, it is sufficient to control the $L^\infty$ norm of $|\xi|^k\partial_t^l\Phi_j(t,s,\xi)$ for~$l+k=0,1$ and~$j=0,1$, which expressions may be obtained by \eqref{eq:Phi0}-\eqref{eq:Phi1}-\eqref{eq:Phi0t}-\eqref{eq:Phi1t}.

In the interval~$I_1$ it holds $\tau\geq\sigma\geq K$, therefore thanks to~\eqref{eq:Hlarge} we get
\[ |\Psi_{k,r,\delta}(t,s,|\xi|)| \lesssim |\xi|^k\,(\Lambda(s)|\xi|)^{-1/2}\,(\Lambda(t)|\xi|)^{-1/2} \,. \]
It immediately follows that
\[ \Psi_{1,\rho-1,1} \,,\quad
|\xi| \Psi_{0,\rho,0} \,, \quad |\xi|^{-1}\,\Psi_{2,\rho-1,0} \,,\quad
\Psi_{1,\rho,-1}\,,\]
are all bounded by $\Lambda(s)^{-1/2}\,\Lambda(t)^{-1/2}$. On the other hand, we can estimate
\[ |\Psi_{0,\rho,0}| \lesssim  |\xi|^{-1}\,\Lambda(s)^{-1/2}\,\Lambda(t)^{-1/2} \lesssim \Lambda(s)^{1/2}\,\Lambda(t)^{-1/2} \,.\]

In the interval~$I_2$ it holds $\tau\geq K\geq\sigma$, therefore thanks to~\eqref{eq:Hlarge} and~\eqref{eq:Hsmall} we get
\begin{align*}
|\Psi_{k,r,\delta}(t,s,|\xi|)| & \lesssim |\xi|^k\,(\Lambda(s)|\xi|)^{-|r|}\,(\Lambda(t)|\xi|)^{-1/2} \,,
\intertext{hence it follows}
|\Psi_{1,\rho-1,1}| & \lesssim |\xi|\,(\Lambda(s)|\xi|)^{-|\rho-1|}\,(\Lambda(t)|\xi|)^{-1/2} \,,\\
|\Psi_{0,\rho,0}| & \lesssim (\Lambda(s)|\xi|)^{-|\rho|}\,(\Lambda(t)|\xi|)^{-1/2}\,,\\
|\xi|\,|\Psi_{1,\rho-1,1}|, \ |\Psi_{2,\rho-1,0}| & \lesssim |\xi|^2 \,(\Lambda(s)|\xi|)^{-|\rho-1|}\,(\Lambda(t)|\xi|)^{-1/2} \,,\\
|\xi| \,|\Psi_{0,\rho,0}|, \ |\Psi_{1,\rho,-1}| & \lesssim |\xi|\,(\Lambda(s)|\xi|)^{-|\rho|}\,(\Lambda(t)|\xi|)^{-1/2} \,.
%
\intertext{Using~$|\xi|^{-1}\lesssim \Lambda(t)$ and~$\mu\geq1$, that is, $\rho\leq 0$, one can estimate}
|\Psi_{1,\rho-1,1}| & \lesssim |\xi|^{-(1/2-\rho)}\,\Lambda(s)^{\rho-1}\,\Lambda(t)^{-1/2} \lesssim \Lambda(s)^{\rho-1}\,\Lambda(t)^{-\rho} \,,\\
|\Psi_{0,\rho,0}| & \lesssim |\xi|^{-(1/2-\rho)} \Lambda(s)^{\rho}\,\Lambda(t)^{-1/2} \lesssim \Lambda(s)^{\rho}\,\Lambda(t)^{-\rho} \,.
\intertext{If~$\mu\geq2$, that is, $\rho\leq -1/2$, then}
|\xi|\,|\Psi_{1,\rho-1,1}|, |\Psi_{2,\rho-1,0}| & \lesssim |\xi|^{\rho+1/2} \,\Lambda(s)^{\rho-1}\,\Lambda(t)^{-1/2} \lesssim \Lambda(s)^{\rho-1}\,\Lambda(t)^{-\rho-1} \,,\\
|\xi| \,|\Psi_{0,\rho,0}|, |\Psi_{1,\rho,-1}| & \lesssim |\xi|^{\rho+1/2}\,\Lambda(s)^{\rho}\,\Lambda(t)^{-1/2} \lesssim \Lambda(s)^{\rho}\,\Lambda(t)^{-\rho-1} \,.
\end{align*}
In the interval~$I_3$ it holds $K\geq\tau\geq\sigma$. We use~\eqref{eq:PsiI} and~\eqref{eq:I}, obtaining
\begin{align*}
|\Psi_{k,r,\delta}(t,s,|\xi|)|
	& \lesssim |\xi|^k\,\left( (\Lambda(s)|\xi|)^{-r}\,(\Lambda(t)|\xi|)^{r+\delta} + (\Lambda(s)|\xi|)^{r}\,(\Lambda(t)|\xi|)^{-(r+\delta)} \right) \\
	& = |\xi|^{k+\delta} \, \Lambda(s)^{-r}\,\Lambda(t)^{r+\delta} + |\xi|^{k-\delta} \, \Lambda(s)^{r}\,\Lambda(t)^{-(r+\delta)}\\
	& \lesssim \Lambda(s)^{-r}\,\Lambda(t)^{r-k} + \Lambda(s)^{r}\,\Lambda(t)^{-r-k} \lesssim \Lambda(s)^{-|r|}\,\Lambda(t)^{|r|-k}\,,
\end{align*}
provided that $k\geq |\delta|$, since $|\xi|\lesssim \Lambda(t)^{-1}$ and $\Lambda(s)\leq \Lambda(t)$. Since~$\rho\leq0$, using $|\xi|\lesssim \Lambda(t)^{-1}$ where needed, it follows again
\begin{align*}
|\Psi_{1,\rho-1,1}| & \lesssim |\xi| \Lambda(s)^{\rho-1}\,\Lambda(t)^{1-\rho} \lesssim \Lambda(s)^{\rho-1}\,\Lambda(t)^{-\rho} \,,\\
|\Psi_{0,\rho,0}| & \lesssim \Lambda(s)^{\rho}\,\Lambda(t)^{-\rho}\,,\\
|\xi|\,|\Psi_{1,\rho-1,1}|, |\Psi_{2,\rho-1,0}| & \lesssim |\xi|^2\,\Lambda(s)^{\rho-1}\,\Lambda(t)^{1-\rho} \lesssim \Lambda(s)^{\rho-1}\,\Lambda(t)^{-\rho-1} \,,\\
|\xi| \,|\Psi_{0,\rho,0}|, |\Psi_{1,\rho,-1}| & \lesssim |\xi|\,\Lambda(s)^{\rho}\,\Lambda(t)^{-\rho}\lesssim \Lambda(s)^{\rho}\,\Lambda(t)^{-\rho-1} \,.
\end{align*}
Using $\Lambda(s) \leq \Lambda(t)$ and $\rho \leq 1/2$, in $I_1$ we also have
\begin{gather*}
\Lambda(s)^{-1/2}\,\Lambda(t)^{-1/2} \leq \Lambda(s)^{\rho-1}\,\Lambda(t)^{-\rho} \,, \\
\Lambda(s)^{1/2}\,\Lambda(t)^{-1/2} \leq \Lambda(s)^{\rho}\,\Lambda(t)^{-\rho}\,.
\end{gather*}
Summarizing and recalling \eqref{eq:Phi0}-\eqref{eq:Phi1}, estimate~\eqref{eq:decayv2} follows. If $\rho\leq-1/2$, that is, $\mu\geq2$, then
\[ \Lambda(s)^{-1/2}\,\Lambda(t)^{-1/2} \leq \Lambda(s)^{\rho}\,\Lambda(t)^{-\rho-1} \,. \]
Recalling \eqref{eq:Phi0}-\eqref{eq:Phi1}-\eqref{eq:Phi0t}-\eqref{eq:Phi1t}, the proof of \eqref{eq:decayenv2} follows.
\end{proof}
\begin{proof}[Proof of Lemma~\ref{Lem:CPlinm}]
We follow the proof of Lemma~\ref{Lem:CPlin2} with some modifications. In $I_1$ we notice that
\begin{gather*}
\frac{\Lambda(t)^\rho}{\Lambda(s)^{\rho-1}} \, \Lambda(s)^{1/2}\,\Lambda(t)^{-1/2} = \Lambda(s)^{\frac\mu2+1} \, \Lambda(t)^{-\frac\mu2} \,,\\
\frac{\Lambda(t)^\rho}{\Lambda(s)^{\rho-1}} \, \Lambda(s)^{-1/2}\,\Lambda(t)^{-1/2} = \Lambda(s)^{\frac\mu2} \, \Lambda(t)^{-\frac\mu2} \,.
\end{gather*}
Moreover, since~$\Lambda(s)\leq\Lambda(t)$ we may estimate
\begin{align*}
\Lambda(s)^{\frac\mu2+1} \, \Lambda(t)^{-\frac\mu2}
    & \leq \Lambda(t)^{-n\left(\frac1m-\frac12\right)}\Lambda(s)^{1+n\left(\frac1m-\frac12\right)} \quad \text{if~$\mu\geq n(2/m-1)$,} \\
\Lambda(s)^{\frac\mu2} \, \Lambda(t)^{-\frac\mu2}
    & \leq \begin{cases}
\Lambda(t)^{-n\left(\frac1m-\frac12\right)}\Lambda(s)^{n\left(\frac1m-\frac12\right)} & \text{if~$\mu\geq n(2/m-1)$,}\\
\Lambda(t)^{-n\left(\frac1m-\frac12\right)-1}\Lambda(s)^{n\left(\frac1m-\frac12\right)+1} & \text{if~$\mu\geq 2+n(2/m-1)$.}
\end{cases}
\end{align*}
Let us define $q\doteq(1/m-1/2)^{-1} \in [2,\infty)$. By virtue of Parseval's identity, we may now estimate
\[ \|v(t,s,\cdot)\|_{L^2} \lesssim \sum_{j=0}^1 \left(\|\Phi_j(t,s,\xi)\|_{L^\infty(I_1)}\, \|v_j(t,s,\cdot)\|_{L^2} + \|\Phi_j(t,s,\xi)\|_{L^q(I_2\cup I_3)}\, \|v_j(t,s,\cdot)\|_{L^m} \right) \,, \]
and similarly for the energy. Let
\[ J_2^\pm \doteq \int_{|\xi|\in I_2} |\xi|^{q\,(\rho\pm1/2)} d\xi\,,\qquad J_3^\pm \doteq \int_{|\xi|\in I_3} |\xi|^{q\,(j+k\pm\delta)} d\xi\,, \]
and~$\eta \doteq \Lambda(t)|\xi|$. It follows
\begin{align*}
J_2^\pm
	& \lesssim \Lambda(t)^{-q\,(\rho\pm1/2)-n} \int_{|\eta|\geq K} |\eta|^{q\,(\rho\pm1/2)} d\eta \lesssim \Lambda(t)^{-q\,(\rho\pm1/2)-n} \,, \\
J_3^\pm
	& \lesssim \Lambda(t)^{-q\,(j+k\pm\delta)-n} \int_{|\eta|\leq K} |\eta|^{q\,(j+k\pm\delta)} d\eta \lesssim \Lambda(t)^{-q\,(j+k\pm\delta)-n} \,,
\end{align*}
provided that $q\,(\rho\pm1/2) < -n$ and that $j+k\pm\delta > -n$. Therefore we obtain
\begin{align*}
\|\Psi_{1,\rho-1,1}\|_{L^q(I_2\cup I_3)}
	& \lesssim \Lambda(s)^{\rho-1}\,\Lambda(t)^{-\rho-n/q}\,, \\
\|\Psi_{0,\rho,0}\|_{L^q(I_2\cup I_3)}
	& \lesssim \Lambda(s)^{\rho}\,\Lambda(t)^{-\rho-n/q}\,,
\intertext{provided that $\rho-1/2<-n/q$, that is, $\mu>2n(1/m-1/2)$, and}
\|(\xi\,\Psi_{1,\rho-1,1},\Psi_{2,\rho-1,0})\|_{L^q(I_2\cup I_3)}
	& \lesssim \Lambda(s)^{\rho-1}\,\Lambda(t)^{-\rho-1-n/q}\,,\\
\| (\xi \,\Psi_{0,\rho,0},\Psi_{1,\rho,-1}) \|_{L^q(I_2\cup I_3)}
	& \lesssim \Lambda(s)^{\rho}\,\Lambda(t)^{-\rho-1-n/q}\,,
\end{align*}
provided that $\rho+1/2<-n/q$, i.e. $\mu>2+2n(1/m-1/2)$. If $\mu= 1 + n(2/m-1)\pm1$, the estimate of~$J_2^\pm$ gives
\[ |J_2^\pm| \leq C_n \left( \log (K/\Lambda(s))-\log (K/\Lambda(t))\right) \,,\]
Combining the estimates for high and low frequencies, we conclude the proof. 
\end{proof}


\section{Proof of Theorems~\ref{Thm:L2} and~\ref{Thm:L1m}}\label{sec:well}

We will use the linear estimates~\eqref{eq:decayvm} and~\eqref{eq:decayenvm} 
to prove~\eqref{eq:uLm} and~\eqref{eq:enLm} for~$\mu>2+n(2/m-1)$. The special case~$\mu=2+n(2/m-1)$ can be easily proved by replacing estimate~\eqref{eq:decayenvm} with~\eqref{eq:decayenvcrit}, whereas estimates~\eqref{eq:uL2} and~\eqref{eq:enL2} follow from~\eqref{eq:decayv2} and~\eqref{eq:decayenv2}.

Using Duhamel's principle and~\eqref{eq:vMult}, a function~$u\in\mathcal{C}([0,\infty),H^1)\cap\mathcal{C}^1([0,\infty),L^2)$ is a solution to~\eqref{eq:diss} if, and only if, it is a fixed point for the operator~$N$ given by
\begin{equation}\label{eq:Nw}
Nu(t,x) = E_0(t,0,x)\ast_{(x)} u_0(x) + E_1(t,0,x)\ast_{(x)} u_1(x) + \int_0^t E_1(t,s,x)\ast_{(x)} f(s,u(s,x))\,ds\,,
\end{equation}
i.e. $Nu(t,\cdot)=u(t,\cdot)$ in~$H^1$ and $\partial_t Nu(t,\cdot)=u_t(t,\cdot)$ in~$L^2$, for any~$t\in[0,\infty)$. For any~$t\geq0$, we consider the spaces
\[ X(t)\doteq \mathcal{C}([0,t],H^1)\cap\mathcal{C}^1([0,t],L^2) \,, \qquad X_0(t)=\mathcal{C}([0,t],H^1)\,,\]
with the norms
\begin{align*}
\|w\|_{X(t)}
    & \doteq \sup_{0\leq \tau\leq t} \Lambda(\tau)^{n(1/m-1/2)} \Bigl( \|w(\tau,\cdot)\|_{L^2} + \Lambda(\tau) \|\nabla w(\tau,\cdot)\|_{L^2} + \lambda(\tau)^{-1}\Lambda(\tau) \|w_t(\tau,\cdot)\|_{L^2} \bigr)\,,\\
\|w\|_{X_0(t)}
    & \doteq \sup_{0\leq \tau\leq t} \Lambda(\tau)^{n(1/m-1/2)} \Bigl( \|w(\tau,\cdot)\|_{L^2} + \Lambda(\tau) \|\nabla w(\tau,\cdot)\|_{L^2} \bigr)\,.
\end{align*}
%
We claim that for any data~$(u_0,u_1)\in \mathcal{D}_m$ the operator~$N$ satisfies the estimates
\begin{align}
\label{eq:well}
\|Nu\|_{X(t)}
    & \leq C\,\|(u_0,u_1)\|_{\mathcal{D}_m} + C\|u\|_{X_0(t)}^p\,, \\
\label{eq:contraction}
\|Nu-N\tilde{u}\|_{X(t)}
    & \leq C\|u-\tilde{u}\|_{X_0(t)} \bigl(\|u\|_{X_0(t)}^{p-1}+\|\tilde{u}\|_{X_0(t)}^{p-1}\bigr)\,,
\end{align}
for any~$u,\tilde{u}\in X(t)$, uniformly with respect to~$t\in [0,\infty)$.
\\
If~\eqref{eq:well} and~\eqref{eq:contraction} hold, then~$N$ maps~$X(t)$ into itself and there exists a unique fixed point~$u\in X(t)$ for the operator~$N$, for sufficiently small data. Indeed, let $\epsilon \,\doteq \|(u_0,u_1)\|_{\mathcal{D}_m}$, and let us define the sequence~$u^{(j)}= Nu^{(j-1)}$ for any~$j\geq1$, with~$u^{(0)}= 0$. Thanks to~\eqref{eq:well}, there exists~$\epsilon_0=\epsilon_0(C)>0$, such that
\begin{equation}\label{eq:recurrence}
\|u^{(j)}\|_{X(t)}\le 2C\epsilon,
\end{equation}
for any $\epsilon \in[0,\epsilon_0]$. Moreover, let us fix~$\epsilon_0(C)$ be such that~$C\epsilon_0^{p-1}<1$. Using~\eqref{eq:contraction} and~\eqref{eq:recurrence}, we obtain
\begin{equation}\label{eq:contr}
\|u^{(j+1)}-u^{(j)}\|_{X(t)} \le C\epsilon^{p-1} \|u^{(j)}-u^{(j-1)}\|_{X(t)} \,,
\end{equation}
therefore $\{u^{(j)} \}$ is a Cauchy sequence in the Banach space $X(t)$, converging to the unique solution of $Nu=u$. Since the constants are independent of~$t$, the global existence follows. The definition of~$\|u\|_{X(t)}$ leads to the decay estimates \eqref{eq:uLm}-\eqref{eq:enLm}.

Therefore, we only need to prove our claims~\eqref{eq:well} and~\eqref{eq:contraction}. During the proof a special role will be played by different applications of Gagliardo-Nirenberg inequality:
\begin{align}
\label{eq:GN}
\|u(s,\cdot)\|_{L^q}^p
    & \lesssim \|u(s,\cdot)\|_{L^2}^{p(1-\theta(q))}\,\|\nabla u(s,\cdot)\|_{L^2}^{p\theta(q)}, \qquad \text{where}\\
\label{eq:thetaq}
\theta(q)
    & \doteq n\left(\frac12-\frac1q\right)\,, \qquad \text{for any} \quad q\in\left[2\,,\ \frac{2n}{n-2}\right]\,.
\end{align}
We prove~\eqref{eq:well}, being the proof of~\eqref{eq:contr} completely analogous.
\begin{proof}[Proof of~\eqref{eq:well}]
From \eqref{eq:decayvm}-\eqref{eq:decayenvm} we derive
\begin{align}
\nonumber
\|Nu(t,\cdot)\|_{L^2}
    & \lesssim \Lambda(t)^{-n(\frac1m-\frac12)}\,\|(u_0,u_1)\|_{L^m\times L^2}\\
\nonumber
    & \quad + \Lambda(t)^{-n(\frac1m-\frac12)}\int_0^t \lambda(s)^{-1}\,\Lambda(s) \,\|f(s,u(s,\cdot))\|_{L^m}ds \\
\label{eq:Nu}
    & \quad + \Lambda(t)^{-n(\frac1m-\frac12)} \int_0^t \lambda(s)^{-1}\,\Lambda(s)^{1+n(\frac1m-\frac12)}\,\|f(s,u(s,\cdot))\|_{L^2}ds\,,\\
\nonumber
\|(\lambda\nabla Nu,\partial_t Nu)(t,\cdot)\|_{L^2}
    & \lesssim \lambda(t)\,\Lambda(t)^{-n(\frac1m-\frac12)-1}\,\|(u_0,u_1)\|_{\mathcal{D}_m}\\
\nonumber
    & \quad + \lambda(t)\,\Lambda(t)^{-n(\frac1m-\frac12)-1}\int_0^t \lambda(s)^{-1}\,\Lambda(s) \,\|f(s,u(s,\cdot))\|_{L^m}ds \\
\label{eq:Nenu}
    & \quad + \lambda(t)\,\Lambda(t)^{-n(\frac1m-\frac12)-1} \int_0^t \lambda(s)^{-1}\,\Lambda(s)^{1+n(\frac1m-\frac12)}\,\|f(s,u(s,\cdot))\|_{L^2}ds\,.
\end{align}
By using~\eqref{eq:disscontr} we can estimate $|f(s,u)|\lesssim \lambda(s)^2\Lambda(s)^\gamma\,|u|^p$. Since $p\geq 2/m$, and~$p\leq n/(n-2)$ if~$n\geq3$, we can apply~\eqref{eq:GN} with~$q=mp$ and~$q=2p$, obtaining
\begin{align}
\label{eq:fusGNlowp}
\||u(s,\cdot)|^p\|_{L^m}
    & \lesssim \|u(s,\cdot)\|_{L^{mp}}^p \lesssim \|u\|_{X_0(s)}^p \Lambda(s)^{-p(n(1/m-1/2)+\theta(mp))}= \|u\|_{X_0(s)}^p \Lambda(s)^{-\frac{n}m(p-1)}\,,\\
\label{eq:fusGNlow2p}
\||u(s,\cdot)|^p\|_{L^2}
    & \lesssim \|u(s,\cdot)\|_{L^{2p}}^p \lesssim \|u\|_{X_0(s)}^p \Lambda(s)^{-p(n(1/m-1/2)+\theta(2p))} = \|u\|_{X_0(s)}^p \Lambda(s)^{-\frac{pn}m+\frac{n}2}\,.
\end{align}
We notice that:
\[ 1+n\left(\frac1m-\frac12\right)-\frac{pn}m+\frac{n}2+\gamma = 1-\frac{n}m(p-1)+\gamma\,, \]
hence
\begin{align}
\nonumber
\|Nu(t,\cdot)\|_{L^2}
    & \lesssim \Lambda(t)^{-n(\frac1m-\frac12)}\,\|(u_0,u_1)\|_{L^m\cap L^2}\\
\label{eq:Nu2}
    & \qquad + \|u\|_{X_0(t)}^p\, \Lambda(t)^{-n(\frac1m-\frac12)}\int_0^t \lambda(s)\,\Lambda(s)^{1-\frac{n}m(p-1)+\gamma}ds \\
\nonumber
\|(\lambda\nabla Nu,\partial_t Nu)(t,\cdot)\|_{L^2}
    & \lesssim \lambda(t)\,\Lambda(t)^{-n(\frac1m-\frac12)-1}\,\|(u_0,u_1)\|_{\mathcal{D}_m}\\
\label{eq:Nenu2}
    & \qquad + \|u\|_{X_0(t)}^p\, \lambda(t)\,\Lambda(t)^{-n(\frac1m-\frac12)-1}\int_0^t \lambda(s)\,\Lambda(s)^{1-\frac{n}m(p-1)+\gamma} \,ds\,.
\end{align}
%
%
Thanks to~\eqref{eq:p}, if we put~$r=\Lambda(s)$ then we get
\[ \int_0^t \lambda(s)\,\Lambda(s)^{1-\frac{n}m(p-1)+\gamma} \,ds = \int_{\Lambda(0)}^{\Lambda(t)} r^{1-\frac{n}m(p-1)+\gamma} \,dr \leq C \,, \]
and this concludes the proof of~\eqref{eq:well}.
\end{proof}
%


\section{Data from a weighted energy space}

If~$f=f(u)$, we may overcome the lower bound~$p\geq2$ in Theorem~\ref{Thm:1L1} if we assume smallness of the initial data in some weighted energy space. Similarly in Theorem~\ref{Thm:L1m} with~$m=1$.
\\
Let~$\lambda(t)$ and~$b(t)$ satisfy Hypothesis~\ref{Hyp:lambda}. For any~$t\geq0$, we define the exponential weight
\begin{equation}\label{eq:omega}
\omega_{(t)}(x)\doteq \exp \left( \frac\mu2 \, \frac{|x|^2}{\Lambda(t)^2} \right) \,,
\end{equation}
and we denote by~$L^2(\omega_{(t)})$ and~$H^1(\omega_{(t)})$ the weighted spaces with norms:
\[ \|u\|_{L^2(\omega_{(t)})}^2 \doteq \int_{\R^n} |u(x)|^2\,\omega_{(t)}^2(x)\,dx\,, \qquad \|u\|_{H^1(\omega_{(t)})}^2 = \|u\|_{L^2(\omega_{(t)})}^2+\|\nabla u\|_{L^2(\omega_{(t)})}^2\,. \]
One may easily check that~$L^2(\omega_{(t)})\hookrightarrow L^1\cap L^2$, for any~$\mu>0$ and~$t\geq0$.
\begin{Thm}\label{Thm:weight}
Let~$n\geq1$, $\mu \geq n+2$. Let~$f(t,u)=\lambda(t)^2f_1(u)$, with~$f_1(u)$ satisfying
%
\[ f_1(0)=0\,, \qquad |f_1(u)-f_1(v)|\lesssim |u-v|(|u|+|v|)^{p-1} \,,\]
%
for some~$p>1+2/n$, and~$p\leq 1+2/(n-2)$ if~$n\geq3$. Then there exists~$\epsilon>0$ such that for any initial data
\begin{equation}\label{eq:dataweight}
(u_0,u_1)\in H^1(\omega_{(0)})\times L^2(\omega_{(0)})\,, \qquad \text{satisfying} \quad \|(u_0,u_1)\|_{H^1(\omega_{(0)})\times L^2(\omega_{(0)})} \leq \epsilon \,,
\end{equation}
there exists a solution~$u$ to~\eqref{eq:diss}. Moreover, $u\in \mathcal{C}([0,\infty),H^1(\omega_{(t)}))\cap\mathcal{C}^1([0,\infty), L^2(\omega_{(t)}))$, and
\begin{align*}
\|u(t,\cdot)\|_{L^2}
    & \lesssim \Lambda(t)^{-\frac{n}2} \, \|(u_0,u_1)\|_{L^2(\omega_{(0)})} \,, \\
\|(\lambda\nabla u,u_t)(t,\cdot)\|_{L^2}
    & \lesssim \lambda(t)\Lambda(t)^{-\frac{n}2-1} \, \|(u_0,u_1)\|_{H^1(\omega_{(0)})\times L^2(\omega_{(0)})} \,,\\
\|u(t,\cdot)\|_{L^2(\omega_{(t)})}
    & \lesssim\lambda(t)\,\Lambda(t)\, \|(u_0,u_1)\|_{H^1(\omega_{(0)})\times L^2(\omega_{(0)})} \,,\\
\|(\lambda\nabla u,u_t)(t,\cdot)\|_{L^2(\omega_{(t)})}
    & \lesssim \lambda(t)\, \|(u_0,u_1)\|_{H^1(\omega_{(0)})\times L^2(\omega_{(0)})} \,.
\end{align*}
\end{Thm}
%
The range of admissible exponents~$p$ for the global existence in Theorem~\ref{Thm:weight} is nonempty for any~$n\geq1$. 
%
\\
If we consider~\eqref{eq:diss1}, then we assume~$f=f(u)$, and the weight is given by
\[ \omega_{(t)}(x) \doteq \exp \left( \frac\mu2 \, \frac{|x|^2}{1+t^2} \right) \,. \]
%
%
By assuming compactly supported data, Y. Wakasugi recently extended the result in~\cite{LNZ} to prove that if~$f(u)=|u|^p$ with~$p>1+2/n$ then there exists~$\overline{\mu}=\overline{\mu}(p,n)$ satisfying $\overline{\mu}(p,n)\approx n^2\,(p-(1+2/n))^{-2}$ such that for any~$\mu\geq\overline{\mu}$ there exists a global solution to~\eqref{eq:diss1}. A loss of information in the decay estimates like $(1+t)^\epsilon$ also appears, where~$\epsilon\approx \mu^{-1}$ (see~\cite{Waka}). We remark that in Theorem~\ref{Thm:weight} we do not require compact support, the threshold is~$\mu\geq n+2$ for any~$p>1+2/n$, and we do not have loss of information in the decay estimates with respect to the linear problem. Moreover, we can deal with a more general propagation speed~$\lambda(t)$.


In order to prove Theorem~\ref{Thm:weight}, we follow the approach in~\cite{DALR, IT}. For the sake of brevity, we only sketch the main ideas, highlighting the differences due to the presence of the propagation speed~$\lambda(t)$.

One can easily prove the local existence of the solution to~\eqref{eq:diss} in
\[ \mathcal{C}\bigl([0,T_{\max}),H^1(\omega_{(t)})\bigr)\cap \mathcal{C}\bigl([0,T_{\max}),L^2(\omega_{(t)})\bigr)\,,\]
for any~$p\leq 1+2/(n-2)$, where by~$T_{\max}>0$ we denote the maximal existence time. Moreover,
\begin{equation}\label{eq:blowup}
\limsup_{t\to T_{\max}} \left( \|u(t,\cdot)\|_{H_1(\omega_{(t)})}^2 + \lambda(t)^{-2}\,\|u_t(t,\cdot)\|_{L^2(\omega_{(t)})}^2\right) = \infty\,,
\end{equation}
if~$T_{\max}<\infty$. Let us define the function
\[ \psi(t,x)\doteq \log \omega_{(t)}(x)= \frac\mu2\,\frac{|x|^2}{\Lambda(t)^2}\,,\]
which has the following property:
\begin{equation}\label{eq:weightprop}
\mu\,\frac{\lambda(t)}{\Lambda(t)}\,\psi_t(t,x) = - |\lambda(t)\nabla \psi(t,x)|^2 \,, \qquad \text{in particular $\psi_t(t,x)\leq0$ since~$\mu\geq0$.}
\end{equation}
We are now in a position to prove the following.
\begin{Lem}\label{Lem:weighten}
Let~$u$ be the local solution to~\eqref{eq:diss}. 
Then for any~$t\in[0,T_{\max})$ and for any~$\varepsilon\in(0,2-2/(p+1))$, the following energy estimate holds:
\begin{align*}
\|(\lambda\nabla u, u_t)(t,\cdot)\|_{L^2(\omega_{(t)})}^2
    & \leq C\, \lambda(t)^2 \,\left(\|(u_0,u_1)\|_{H^1(\omega_{(0)})\times L^2(\omega_{(0)})}^2 + \|(u_0,u_1)\|_{H^1(\omega_{(0)})\times L^2(\omega_{(0)})}^{\frac{p+1}2}\right) \\
    & \qquad + C_\varepsilon\,\lambda(t)^2 \sup_{s\in[0,t]} \Bigl( \Lambda(s)^\varepsilon \|e^{(\varepsilon+2/(p+1))\psi(s,\cdot)} u(s,\cdot)\|_{L^{p+1}} \Bigr)^{p+1} \,.
\end{align*}
\end{Lem}
\begin{proof}
We recall that $f(t,u)=\lambda(t)^2f_1(u)$ in Theorem~\ref{Thm:weight}. If we define the functional
\[ G(t) \doteq \frac1{\lambda(t)^2}\,\|(\lambda\nabla u, u_t)(t,\cdot)\|_{L^2(\omega_{(t)})}^2 - \int_{\R^n} F(u)\,dx \,,\quad \text{where} \ F(u) \doteq \int_0^u f_1(v)\,dv \,,\]
then it follows that
\begin{equation}\label{eq:Gfunct}
G(t)-G(0) \leq -4 \int_0^t \int_{\R^n} \psi_t(s,x)\,e^{2\psi(s,x)}\,F(u(s,x))\,dx\,ds\,.
\end{equation}
Indeed, we have:
\begin{multline*}
\partial_t \left( \frac{e^{2\psi}}2\, \bigl(\lambda(t)^{-2}\,|u_t|^2+|\nabla u|^2-F(u)\bigr) \right) = \nabla\cdot(e^{2\psi}u_t\nabla u) +\lambda(t)^{-2}\psi_te^{2\psi}u_t^2 \\
    + \frac{e^{2\psi}}{\psi_t}|u_t\nabla\psi-\psi_t\nabla u|^2 -\lambda(t)^{-2}\frac{e^{2\psi}}{\psi_t}u_t^2\bigl((b(t)+\lambda'(t)/\lambda(t))\psi_t+|\nabla\psi|^2\bigr)-2\psi_te^{2\psi}F(u)\,.
\end{multline*}
By using divergence theorem and~\eqref{eq:weightprop}, the proof of~\eqref{eq:Gfunct} follows. By using Sobolev embedding, we get
\[ G(0) \lesssim \|(u_0,u_1)\|_{H^1(\omega_{(0)})\times L^2(\omega_{(0)})}^2 + \|(u_0,u_1)\|_{H^1(\omega_{(0)})\times L^2(\omega_{(0)})}^{\frac{p+1}2} \,. \]
Estimating
\[ |\psi_t(s,x)|e^{-\varepsilon(p+1)\psi(s,x)} = 2\,\frac{\lambda(t)}{\Lambda(t)} \psi(s,x)e^{-\varepsilon(p+1)\psi(s,x)} \leq C_\varepsilon\,\frac{\lambda(t)}{\Lambda(t)}\,, \qquad \text{and} \quad \int_0^t \frac{\lambda(s)}{\Lambda(s)^{1+\varepsilon}}\,ds \leq C_\varepsilon\,, \]
and~$|F(u(s,x))| \lesssim |u(s,x)|^{p+1}$ we may conclude the proof.
\end{proof}
The advantage of working with weighted spaces relies in the chance to estimate
\begin{equation}\label{eq:fus1weight}
\|f_1(u(s,\cdot))\|_{L^1} \lesssim \|u(s,\cdot)\|_{L^p}^p \lesssim \Lambda(s)^{\frac{n}2}\,\|e^{\varepsilon\psi(s,\cdot)}u(s,\cdot)\|_{L^{2p}}^p \,,
\end{equation}
by using H\"older inequality and
\[ \int_{\R^n} e^{-\frac{c|x|^2}{\Lambda(s)^2}} dx = \Lambda(s)^n\,\int_{\R^n} e^{-c|y|^2} dy \lesssim \Lambda(s)^n\,. \]
Trivially, we may also estimate
\begin{equation}\label{eq:fus2weight}
\|f_1(u(s,\cdot))\|_{L^2} \lesssim \|e^{\varepsilon\psi(s,\cdot)}u(s,\cdot)\|_{L^{2p}}^p \,.
\end{equation}
\begin{proof}[Proof of Theorem~\ref{Thm:weight}]
By contradiction, let us assume that for any~$\epsilon>0$ there exist data satisfying~\eqref{eq:dataweight} such that the solution to~\eqref{eq:diss} is not global, that is, $T_{\max}<\infty$. Similarly to the proof of Theorem~\ref{Thm:L1m}, for any~$t\in(0,T_{\max})$ we may consider the space
\begin{align}
\nonumber
X(t)
    & \doteq \mathcal{C}\bigl([0,t],H^1(\omega_{(\tau)})\bigr)\cap \mathcal{C}\bigl([0,t],L^2(\omega_{(\tau)})\bigr)\,,\qquad \text{with norm}\\
\label{eq:normweightA}
\|u\|_{X(t)}
    & \doteq \max_{\tau\in[0,t]} \Bigl( \lambda(\tau)^{-1}\|(\lambda\nabla u, u_t)(\tau,\cdot)\|_{L^2(\omega_{(\tau)})} \Bigr. \\
\label{eq:normweightB}
    & \qquad\qquad \Bigl. + \lambda(\tau)^{-1}\Lambda(\tau)^{\frac{n}2+1} \|(\lambda\nabla u,u_t)(\tau,\cdot)\|_{L^2} + \Lambda(\tau)^{\frac{n}2} \|u(\tau,\cdot)\|_{L^2} \Bigr)\,.
\end{align}
We may immediately use Lemma~\ref{Lem:weighten} to estimate the weighted energy in~\eqref{eq:normweightA}. On the other hand, using the linear estimates in Lemma~\ref{Lem:CPlinm} as we did in the proof of Theorem~\ref{Thm:L1m}, together with \eqref{eq:fus1weight}-\eqref{eq:fus2weight}, we can control the terms in~\eqref{eq:normweightB}, obtaining:
\begin{align}
\nonumber
\|u\|_{X(t)} \lesssim \epsilon+\epsilon^{\frac{p+1}2}
    & + \sup_{\tau\in[0,t]} \left(\Lambda(\tau)^\varepsilon \, \|e^{(\varepsilon+2/(p+1))\psi(\tau,\cdot)} u(\tau,\cdot)\|_{L^{p+1}}\right)^{\frac{p+1}2} \\
\label{eq:weightpreliminar}
    & + \sup_{\tau\in[0,t]} \left(\Lambda(\tau)^{\frac{n}2+\varepsilon}\, \|e^{\varepsilon\psi(\tau,\cdot)} u(\tau,\cdot)\|_{L^{2p}}\right)^p\,.
\end{align}
In order to manage the last two terms we use a Gagliardo-Nirenberg type inequality (see Lemma 2.3 in~\cite{IT} and Lemma~9 in~\cite{DALR}) and we get
\begin{equation}\label{eq:GNweight}
\|e^{\sigma\psi(t,\cdot)}v\|_{L^q} \leq C_\sigma\,\Lambda(t)^{1-\theta(q)}\, \|\nabla v\|^{1-\sigma}_{L^2}\,\|e^{\psi(t,\cdot)}\nabla v\|^\sigma_{L^2}\,,
\end{equation}
for any~$\sigma\in [0,1]$ and~$v\in H^1_{\sigma\psi(t,\cdot)}$, where~$\theta(q)$ is as in~\eqref{eq:thetaq}. By using~\eqref{eq:GNweight}, it follows
\begin{align}
\label{eq:GNp1}
\|e^{(\varepsilon+2/(p+1))\psi(\tau,\cdot)}u(\tau,\cdot)\|_{L^{p+1}} & \leq \|u\|_{X(t)} \, \Lambda(\tau)^{1-\theta(p+1)-(1-2/(p+1)-\varepsilon)(n/2+1)}\,,\\
\label{eq:GN2p}
\|e^{\varepsilon\psi(\tau,\cdot)}u(\tau,\cdot)\|_{L^{2p}} & \leq \|u\|_{X(t)} \, \Lambda(\tau)^{1-\theta(2p)-(1-\varepsilon)(n/2+1)}\,.
\end{align}
We remark that~$2< p+1 < 2p\leq 2n/(n-2)$, hence Gagliardo-Nirenberg inequality is applicable. Since~$p>1+2(2+\gamma)/n$, it follows that
\[ 1-\theta(p+1) -(1-2/(p+1))(n/2+1) = 1-\theta(2p)-(n/2+1) = \frac{1-(p-1)n/2}p <0\,. \]
Therefore, if we take~$\varepsilon>0$ sufficiently small, from~\eqref{eq:weightpreliminar} we may obtain
\[ \|u\|_{X(t)} \lesssim \epsilon+\epsilon^{\frac{p+1}2}+\|u\|_{X(t)}^{\frac{p+1}2}+\|u\|_{X(t)}^p\,,\]
uniformly with respect to~$t\in [0,T_{\max})$. By standard arguments, it follows that~$\|u\|_{X(t)}$ is bounded with respect to~$t\in[0,T_{\max})$, provided that~$\epsilon>0$ is sufficiently small. Hence~$\|u(t,\cdot)\|_{L^2(\omega_{(t)})}$ is bounded too. This contradicts~\eqref{eq:blowup}, hence the maximal existence time is~$T_{\max}=\infty$.
\end{proof}




\appendix

\section{Linear estimates under the threshold~$\mu=2$}

If $\mu\in(0,2)$ then the $L^2-L^2$ estimate of the energy of the solution to the linear problem~\eqref{eq:CPlin} is worse than~\eqref{eq:decayenv2}, since the dissipation becomes \emph{non effective} and we get
\begin{equation}\label{eq:decayenvsmall}
\|(\lambda\nabla v,v_t)(t,\cdot)\|_{L^2}\lesssim \lambda(t)\,\Lambda(t)^{-\frac\mu2}\,\Lambda(s)^{\frac\mu2} \, \left(\|v_0\|_{H^1}+\frac1{\lambda(s)}\|v_1\|_{L^2}\right) \,.
\end{equation}
Indeed, we may follow the proof of Lemma~\ref{Lem:CPlinm}, but now~$\rho\in(-1/2,1/2)$. The estimate in $I_1$ remains the same. In $I_2$, using $|\xi|\lesssim \Lambda(s)^{-1}$, we get
\begin{align*}
|\xi|\,|\Psi_{1,\rho-1,1}|, \ |\Psi_{2,\rho-1,0}| & \lesssim |\xi|^{\rho+1/2} \,\Lambda(s)^{\rho-1}\,\Lambda(t)^{-1/2} \lesssim \Lambda(s)^{-3/2}\,\Lambda(t)^{-1/2} \,, \\
|\xi| \,|\Psi_{0,\rho,0}|, \ |\Psi_{1,\rho,-1}| & \lesssim \begin{cases}
|\xi|^{\rho+1/2}\,\Lambda(s)^{\rho}\,\Lambda(t)^{-1/2} \lesssim \Lambda(s)^{-1/2}\,\Lambda(t)^{-1/2} \,, & \text{if $\mu\in(1,2)$,}\\
|\xi|^{1/2-\rho}\,\Lambda(s)^{-\rho}\,\Lambda(t)^{-1/2} \lesssim \Lambda(s)^{-1/2}\,\Lambda(t)^{-1/2} \,, & \text{if $\mu\in(0,1)$,}
\end{cases}
\end{align*}
If~$\rho\in(0,1/2)$, i.e. $\mu\in(0,1)$, 
using~$|\xi|\lesssim \Lambda(t)^{-1}$, we derive
\begin{align*}
|\xi| \,|\Psi_{0,\rho,0}| \lesssim |\xi|\,\Lambda(s)^{-\rho}\,\Lambda(t)^{\rho}
	& \lesssim \Lambda(s)^{-\rho}\,\Lambda(t)^{\rho-1}  \,,\\
|\Psi_{1,\rho,-1}|
	& \lesssim \Lambda(s)^{-\rho}\,\Lambda(t)^{\rho-1}  \,,
\end{align*}
in the interval~$I_3$. Since $|\rho|-1 \leq -1/2$, the worst rate for $|\xi|\,|\Psi_{1,\rho-1,1}|$, $|\Psi_{2,\rho-1,0}|$, $|\xi| \,|\Psi_{0,\rho,0}|$ and $|\Psi_{1,\rho,-1}|$ is now given by $\Lambda(t)^{-1/2}$, therefore, due to
\[ \frac{\Lambda(t)^\rho}{\Lambda(s)^{\rho-1}} \, \Lambda(s)^{-1/2}\,\Lambda(t)^{-1/2} = \Lambda(s)^{\frac\mu2} \,\Lambda(t)^{-\frac\mu2} \,, \]
estimates~\eqref{eq:decayenvsmall} follows. Estimate~\eqref{eq:decayenvsmall} is consistent with the energy estimate proved in Example 3 in~\cite{DAE} for~$s=0$ and~$\mu\in[0,2]$.
\\
One may immediately use estimate~\eqref{eq:decayenvsmall} to extend Theorem~\ref{Thm:L2} to the case~$\mu\in[1,2)$, modifying the proof where needed.
\begin{Rem}\label{Rem:L2small}
Let~$n\geq1$. If~$\mu\in[1,2)$ and
\begin{equation}\label{eq:p2small}
p>1+\frac{4(2+\gamma)}{\mu\,n}\,,
\end{equation}
then there exists~$\epsilon>0$ such that for any initial data satisfying~\eqref{eq:dataL2} there exists a solution to~\eqref{eq:diss}. Moreover, the solution satisfies~\eqref{eq:uL2} and its energy satisfies the estimate
\begin{equation}
\label{eq:enL2small}
\|(\lambda\nabla u,u_t)(t,\cdot)\|_{L^2} \lesssim \lambda(t)\,\Lambda(t)^{-\frac\mu2} \, \|(u_0,u_1)\|_{H^1\times L^2} \,.
\end{equation}
%
However, we do not expect condition~\eqref{eq:p2small} to be optimal. Indeed, for~$\mu\in(0,2)$ the model becomes more \emph{hyperbolic} hence the use of linear $L^2-L^2$ estimates which are analogous to the corresponding heat equation is not meaningful (see~\cite{W04}).
\end{Rem}
A different effect appears if we are interested in estimates of the solution to~\eqref{eq:CPlin}, for~$\mu\in(0,1)$. It is convenient to separate contributions coming from~$v_0$ and~$v_1$. Let~$v_1\equiv0$. If~$v_0\in H^1$ or~$v_0\in L^m\cap H^1$, we still have estimates~\eqref{eq:decayv2} for any~$\mu\geq0$, estimate~\eqref{eq:decayvm} for~$\mu>n(2/m-1)$ and estimate~\eqref{eq:decayvcrit} for~$\mu=n(2/m-1)$. Otherwise, the estimate rate with respect to~$t$ becomes worse.
\begin{Lem}\label{Lem:vsmall}
Let~$\mu\in(0,1)$ and~$v_0\equiv0$. If $v_1\in L^2$ then the solution to~\eqref{eq:CPlin} satisfies the estimate
\begin{equation}
\label{eq:decayvsmall2}
\|v(t,\cdot)\|_{L^2} \lesssim \Lambda(t)^{1-\mu}\, \frac{\Lambda(s)^\mu}{\lambda(s)} \, \|v_1\|_{L^2} \,.
\end{equation}
If $v_1\in L^m\cap L^2$ for some~$m\in[1,2)$ and~$\mu < 2 -n(2/m-1)$, then the solution to~\eqref{eq:CPlin} satisfies the estimate
\begin{align}
\label{eq:decayvsmall}
\|v(t,\cdot)\|_{L^2}
    & \lesssim \Lambda(t)^{(1-\mu)-n\left(\frac1m-\frac12\right)}\,\frac{\Lambda(s)^\mu}{\lambda(s)} \left( \|v_1\|_{L^m} + \Lambda(s)^{n\left(\frac1m-\frac12\right)} \, \|v_1\|_{L^2} \right) \,,
\intertext{whereas if~$\mu = 2 -n(2/m-1)$, it satisfies the estimate}
\label{eq:decayvsmallcrit}
\|v(t,\cdot)\|_{L^2}
    & \lesssim \Lambda(t)^{-\frac\mu2}\, \frac1{\lambda(s)}\,\log\left(1+\frac{\Lambda(t)}{\Lambda(s)}\right)\, \left( \Lambda(s)^\mu\,\|v_1\|_{L^m} + \Lambda(s)^{1+\frac\mu2} \, \|v_1\|_{L^2} \right) \,,
\end{align}
\end{Lem}
\begin{proof}
We only prove~\eqref{eq:decayvsmall}, being the other two estimates similar. We follow the proof of Lemma~\ref{Lem:CPlinm}, but now~$\rho\in(0,1/2)$. The estimate in~$I_1$ remains the same. In $I_2$ we may estimate
\[ |\Psi_{0,\rho,0}| \lesssim |\xi|^{-\rho-1/2} \Lambda(s)^{-\rho}\Lambda(t)^{-1/2} \,,\]
therefore, using~$q(-\rho-1/2)<-n$, that is, $\mu < 2 -n(2/m-1)$, we derive
\[ \int_{|\xi|\in I_2} |\xi|^{-q(\rho+1/2)}\,d\xi \lesssim \Lambda(t)^{q(\rho+1/2)-n}\,. \]
%
%
On the other hand, in $I_3$ we may estimate $|\Psi_{0,\rho,0}|\lesssim \Lambda(s)^{-\rho}\Lambda(t)^{\rho}$, therefore
\[ \int_{|\xi|\in I_3} 1\,d\xi \lesssim \Lambda(t)^{-n}\,. \]
%
%
%
Summarizing, we proved
\[ \|\Psi_{0,\rho,0}\|_{L^q(I_2\cap I_3)}\lesssim \Lambda(s)^{-\rho}\,\Lambda(t)^{\rho-n/q}\,,\]
%
hence estimate~\eqref{eq:decayvsmall} follows.
\end{proof}


\section{Additional considerations in one space dimension} 

In this Appendix we fix~$n=1$.

If~$\mu\in[2,3)$, according to Corollary~\ref{Cor:ell}, if data are small in~$\mathcal{D}_\ell$ then we have global existence for any~$p\geq \mu-1$ satisfying~\eqref{eq:pell}, i.e.
\[ p > 1 + \frac{2(2+\gamma)}{\mu-1}\,. \]
If~$\mu\in[1,2)$, according to Remark~\ref{Rem:L2small}, if data are small in~$H^1\times L^2$ then we have global existence for any
\[ p > 1 + \frac{4(2+\gamma)}\mu\,.\]
However, we may improve this lower bound for~$p$ if data are small in~$\mathcal{D}_1$.
\begin{Cor}\label{Cor:mix}
Let~$n=1$, $\mu\in[1,3)$ and~$p\geq2$, satisfying
\begin{equation}\label{eq:p12}
p > 1 + \frac{4(2+\gamma)}{\mu+1}\,.
\end{equation}
Then for any initial data satisfying~\eqref{eq:dataL1} there exists a solution to~\eqref{eq:diss}. Moreover, estimate~\eqref{eq:uLm} with~$m=1$ holds for the solution, together with
\begin{equation}\label{eq:en1ell}
\|(\lambda\nabla u,u_t)(t,\cdot)\|_{L^2}\lesssim \begin{cases}
\lambda(t)\,\Lambda(t)^{-\frac\mu2}\log(e+\Lambda(t))\, \|(u_0,u_1)\|_{\mathcal{D}_1} & \text{if~$\mu\in(2,3)$,}\\
\lambda(t)\,\Lambda(t)^{-\frac\mu2}\, \|(u_0,u_1)\|_{\mathcal{D}_1} \,, & \text{if~$\mu\in[1,2]$,}
\end{cases}
\end{equation}
for its energy.
\end{Cor}
We remark that the exponent in~\eqref{eq:p12} is lower than the one in~\eqref{eq:pell} for any~$\mu\in[2,3)$, and it is lower than the one in~\eqref{eq:p2small} for any~$\mu\in[1,2)$. This improvement does not appear in space dimension~$n\geq2$, if one extends this strategy.
\begin{proof}
We prove for~$\mu\in(2,3)$, being the case~$\mu\in[1,2]$ analogous and simpler. We follow the proof of Theorem~\ref{Thm:L1m} but we consider the norm on~$X_0(t)$ given by
\[ \|w\|_{X_0(t)} \doteq \sup_{0\leq \tau\leq t} \bigl( \Lambda(\tau)^{\frac12}\|w(\tau,\cdot)\|_{L^2} + \Lambda(\tau)^{\frac\mu2} \bigl(\log (e+\Lambda(\tau))\bigr)^{-1}\|\nabla w(\tau,\cdot)\|_{L^2} \bigr)\,,\]
and similarly the norm on~$X(t)$. Using~\eqref{eq:GN}, we may estimate
\begin{align}
\nonumber
\|u(s,\cdot)\|_{L^q}
	& \lesssim \|u\|_{X_0(s)}^p \Lambda(s)^{-(1-\theta(q))\frac{n}2-\theta(q)\frac\mu2}\bigl(\log (e+\Lambda(\tau))\bigr)^{\theta(q)}\\
\label{eq:GNmix}
	& = \Lambda(s)^{-\frac12-\theta(q)\frac{\mu-1}2} \bigl(\log (e+\Lambda(\tau))\bigr)^{\theta(q)} \,,
\end{align}
for~$q=p, \ell p, 2p$, that is, \eqref{eq:fusGNlowp}-\eqref{eq:fusGNlow2p} are replaced by
\begin{align*}
\|f(u(s,\cdot))\|_{L^1}
    & \lesssim \|u\|_{X_0(s)}^p \Lambda(s)^{\gamma-p(\frac1m-\frac12+\theta(p)\frac{\mu-1}2)}= \|u\|_{X_0(s)}^p \Lambda(s)^{\gamma-p\frac14(\mu+1)+\frac12(\mu-1)}\,,\\
\|f(u(s,\cdot))\|_{L^\ell}
    & \lesssim \|u\|_{X_0(s)}^p \Lambda(s)^{\gamma-p(\frac1m-\frac12+\theta(\ell p)\frac{\mu-1}2)}= \|u\|_{X_0(s)}^p \Lambda(s)^{\gamma-p\frac14(\mu+1)+\frac1{2\ell}(\mu-1)}\,,\\
\|f(u(s,\cdot))\|_{L^2}
    & \lesssim \|u\|_{X_0(s)}^p \Lambda(s)^{\gamma-p(\frac1m-\frac12+\theta(2p)\frac{\mu-1}2)} = \|u\|_{X_0(s)}^p \Lambda(s)^{\gamma-p\frac{n}4(\mu+1)+\frac14(\mu-1)}\,.
\end{align*}
Let us put
\[ p_r \doteq p\,\frac14\,(\mu+1)-\frac1{2r}\,(\mu-1) \,, \qquad r=1,\ell,2\,. \]
Using \eqref{eq:decayvm} with~$m=1$ and \eqref{eq:decayenvcrit} with~$m=\ell$ we obtain
\begin{align}
\nonumber
\|Nu(t,\cdot)\|_{L^2}
    & \lesssim \Lambda(t)^{-\frac12}\,\|(u_0,u_1)\|_{L^m\cap L^2}\\
\label{eq:Nu2mix}
    & \quad + \|u\|_{X_0(t)}^p\, \Lambda(t)^{-\frac12}\int_0^t \lambda(s)\,\Lambda(s)^{1+\gamma-p_1}\bigl(\log (e+\Lambda(\tau))\bigr)^p ds \\
\label{eq:Nu3mix}
    & \quad + \|u\|_{X_0(t)}^p\, \Lambda(t)^{-\frac12} \int_0^t \lambda(s)\,\Lambda(s)^{1+\frac12+\gamma-p_2}\bigl(\log (e+\Lambda(\tau))\bigr)^p\,ds\,,\\
\nonumber
\|\nabla Nu(t,\cdot)\|_{L^2}
    & \lesssim \Lambda(t)^{-\frac\mu2} \log(e+\Lambda(t))\,\|(u_0,u_1)\|_{\mathcal{D}_m}\\
\label{eq:Nux2mix}
    & \quad + \|u\|_{X_0(t)}^p\, \Lambda(t)^{-\frac\mu2} \log(e+\Lambda(t)) \int_0^t \lambda(s)\,\Lambda(s)^{1+\gamma-p_\ell}\bigl(\log (e+\Lambda(\tau))\bigr)^p \,ds \\
\label{eq:Nux3mix}
    & \quad + \|u\|_{X_0(t)}^p\, \Lambda(t)^{-\frac\mu2} \int_0^t \lambda(s)\,\Lambda(s)^{\frac\mu2+\gamma-p_2}\bigl(\log (e+\Lambda(\tau))\bigr)^p\,ds\,,
\end{align}
and similarly for~$\partial_t Nu$. We notice that
\[ p_\ell > p_1 > p_2 - \frac12 \,, \qquad \text{and that} \quad p_2 +1 - \frac\mu2 > p_2 -\frac12 = (p-1) \frac14 (\mu+1) \,, \]
therefore the integrals in \eqref{eq:Nu2mix}-\eqref{eq:Nu3mix}-\eqref{eq:Nux2mix}-\eqref{eq:Nux3mix} are bounded if, and only if, $(p_2-1/2)>2+\gamma$, that is, \eqref{eq:p12}.
\end{proof}
We remark that in space dimension~$n=1$ the classical semilinear wave equation $u_{tt}-\triangle u=|u|^p$ admits no global solution, for any~$p>1$. Therefore, we still have concrete benefits from the damping term, even below the threshold~$\mu=2$. Moreover, if~$\mu\in(0,1]$, one may use the linear estimate~\eqref{eq:decayvsmallcrit} to obtain global existence by assuming smallness of the initial data in~$\mathcal{D}_\kappa$, where
\[ \kappa(\mu) \doteq \frac2{3-\mu}\,, \]
for any~$p\geq 4/(3-\mu)$ such that
\begin{equation}\label{eq:phyp}
p > 1 + \frac{2(2+\gamma)}\mu\,.
\end{equation}
In~\cite{Waka} it is proved that if~$\mu\in(0,1)$ and~$f=f(u)=|u|^p$, then there exists no global solution to~\eqref{eq:diss1} for any
\begin{equation}\label{eq:wakablow}
1 < p \leq 1+\frac2{n-(1-\mu)}\,,
\end{equation}
provided that~$u_1\in L^1$ and
\[ \int_{\R^n} u_1(x)\,dx >0\,. \]
We notice that the exponent in~\eqref{eq:wakablow} tends to Fujita exponent~$1+2/n$ as~$\mu\to1$ and to Kato exponent~$1+2/(n-1)$ (see~\cite{Kato, strauss}) as~$\mu\to0$. This effect is related to the loss of parabolic properties of the equation in~\eqref{eq:diss1} as~$\mu$ becomes smaller, in particular under the threshold~$\mu=1$. Following the proof of Theorem~1.4 in~\cite{Waka}, condition~\eqref{eq:wakablow} can be easily extended to
\[ 1 < p \leq 1+\frac{2+\gamma}{n-(1-\mu)}\,. \]
if~$f(t,u)\gtrsim (1+t)^\gamma |u|^p$. This exponent gives~$1+(2+\gamma)/\mu$ in space dimension~$n=1$. Still, there exists a gap between the exponents in~\eqref{eq:phyp} and~\eqref{eq:wakablow}. The problem to cover this gap remains open.


\begin{thebibliography}{DALR}
\bibitem{DAEm} M. D'Abbicco, M.R. Ebert: \emph{Hyperbolic-like estimates for higher order equations}, J. Math. Anal. Appl. \textbf{395} (2012), 747--765, doi: 10.1016/j.jmaa.2012.05.070.
\bibitem{DAE} M. D'Abbicco, M.R. Ebert, \emph{A Class of Dissipative Wave Equations with Time-dependent Speed and Damping}, J. Math. Anal. Appl. \textbf{399} (2012), 315--332, doi:10.1016/j.jmaa.2012.10.017.
\bibitem{DAL} M. D'Abbicco, S. Lucente, \emph{A modified test function method for damped wave equations}, arXiv: 1211.0453, 22 pp.
\bibitem{DALR} M. D'Abbicco, S. Lucente, M. Reissig, \emph{Semilinear wave equations with effective damping}, arXiv: 1210.3493, 28 pp.
\bibitem{DAR} M. D'Abbicco, M. Reissig: \emph{Long time asymptotics for 2 by 2 hyperbolic systems}, J. Differential Equations \textbf{250} (2011), 752--781 doi: 10.1016/j.jde.2010.08.001.
\bibitem{DAR2} M. D'Abbicco, M. Reissig:  \emph{Blow-up of the Energy at Infinity for 2 by 2 Systems}, J. Differential Equations \textbf{252} (2012), 477--504, doi: 10.1016/j.jde.2011.08.033.
\bibitem{Fuj} H. Fujita, \emph{On the blowing up of solutions of the Cauchy Problem for $u_t=\triangle u+u^{1+\alpha}$}, J. Fac.Sci. Univ. Tokyo \textbf{13}
(1966), 109--124.
\bibitem{H-W} F. Hirosawa, J. Wirth, \emph{Generalised energy conservation law for wave equations with variable propagation speed}, J. Math. Anal. Appl. \textbf{358} (2009), 56--74.
\bibitem{IMN} R. Ikehata, Y. Mayaoka, T. Nakatake, \emph{Decay estimates of solutions for dissipative wave equations in $\R^N$ with lower power nonlinearities}, J. Math. Soc. Japan, \textbf{56} (2004), 
365--373.
\bibitem{IO} R. Ikehata, M. Ohta, \emph{Critical exponents for semilinear dissipative wave equations in $\R^N$}, J. Math. Anal. Appl. 269 (2002),
87--97.
\bibitem{IT} R. Ikehata, K. Tanizawa, \emph{Global existence of solutions for semilinear damped wave equations in $R^N$ with noncompactly supported initial data}, Nonlinear Analysis \textbf{61} (2005), 1189--1208.
\bibitem{ITY} R. Ikehata, G. Todorova, B. Yordanov, \emph{Critical exponent for semilinear wave equations with a subcritical potential}, Funkcial. Ekvac. 52 (2009), 411--435.
\bibitem{ITY11} R. Ikehata, G. Todorova, B. Yordanov, \emph{Optimal Decay Rate of the Energy for Wave Equations with Critical Potential}, J. of the Mathematical Society of Japan, in press, http://mathsoc.jp/publication/JMSJ/pdf/JMSJ6143.pdf
\bibitem{Kato} T. Kato, \emph{Blow-up of solutions of some nonlinear hyperbolic equations}, Commun. Pure Appl. Math., \textbf{33} (1980), 501--505.
\bibitem{LNZ11} J. Lin, K. Nishihara, J. Zhai, \emph{Decay property of solutions for damped wave equations with space-time dependent damping term}, J. Math. Anal. Appl. \textbf{374} (2011), 602--614.
\bibitem{LNZ} J. Lin, K. Nishihara, J. Zhai, \emph{Critical exponent for the semilinear wave equation with time-dependent damping}, Discrete and Continuous Dynamical Systems, \textbf{32}, no.12 (2012), 4307--4320, doi:10.3934/dcds.2012.32.4307.
\bibitem{Matsu} A. Matsumura, \emph{On the asymptotic behavior of solutions of semi-linear wave equations}, Publ. RIMS. \textbf{12} (1976), 169--189.
\bibitem{Nx} K. Nishihara, \emph{Decay properties for the damped wave equation with space dependent potential and absorbed semilinear term}, Commun. Partial Differential Equations \textbf{35} (2010), 1402--1418.
\bibitem{N} K. Nishihara, \emph{Asymptotic behavior of solutions to the semilinear wave equation with time-dependent damping}, Tokyo J. of Math. \textbf{34} (2011), 327--343.
\bibitem{NO} M. Nakao, K. Ono, \emph{Existence of global solutions to the Cauchy problem for the semilinear dissipative wave equations}, Math. Z. 214 (1993), 325--342.
\bibitem{strauss} W.A. Strauss, \emph{Nonlinear wave equations}, CBMS Regional Conference Series in Mathematics, \textbf{73}, Amer. Math. Soc. Providence, RI, 1989.
\bibitem{TY} G. Todorova, B. Yordanov, \emph{Critical Exponent for a Nonlinear Wave Equation with Damping}, Journal of Differential Equations \textbf{174} (2001), 464--489.
\bibitem{Wakatx} Y. Wakasugi: \emph{Small data global existence for the semilinear wave equation with space-time dependent damping}, arXiv:1202.5379v1, 21 pp.
\bibitem{Waka} Y. Wakasugi: \emph{Critical exponent for semilinear wave equation with scale invariant damping}, arXiv:1211.2900, 13 pp.
\bibitem{W04} J. Wirth: \emph{Solution representations for a wave equation with weak dissipation}, Math. Meth. Appl. Sci. \textbf{27}, 101--124 (2004), doi: 10.1002/mma.446.
\bibitem{W05} J. Wirth, \emph{Asymptotic properties of solutions to wave equations with time-dependent dissipation}, PhD Thesis, TU Bergakademie Freiberg, 2004.
\bibitem{W06} J. Wirth, \emph{Wave equations with time-dependent dissipation I. Non-effective dissipation}, J. Diff. Eq. \textbf{222} (2006), 487--514
\bibitem{W07} J. Wirth, \emph{Wave equations with time-dependent dissipation II. Effective dissipation}, J. Differential Equations \textbf{232} (2007), 74--103.
\bibitem{Ya} T. Yamazaki \emph{Asymptotic behavior for abstract wave equations with decaying dissipation}, Adv. in Diff. Eq. \textbf{11} (2006), 419--456.
\bibitem{Ya1} T. Yamazaki, \emph{Diffusion phenomenon for abstract wave equations with decaying dissipation}, Adv. Stud. Pure Math. \textbf{47} (2007), 363--381.
\bibitem{Z} Qi S. Zhang, \emph{A blow-up result for a nonlinear wave equation with damping: the critical case}, C. R. Acad. Sci. Paris S\'{e}r. I Math. \textbf{333} (2001), 109--114.
\end{thebibliography}
\end{document}